\documentclass[11pt]{article}
\usepackage{amsmath,epsf,cite}
\usepackage{amssymb,amsthm,tocvsec2}
\usepackage{graphicx}
\usepackage{sidecap}
\usepackage{wrapfig}
\usepackage{pict2e}
\usepackage{float}
\usepackage{setspace}
\usepackage{float}
\usepackage{color}
\usepackage{mathrsfs}
\usepackage{epsfig,latexsym,subfigure}
\usepackage{mdwlist}
\usepackage{paralist}
\usepackage{setspace}
\usepackage{tikz}
\usetikzlibrary{shapes,arrows}
\usepackage{enumitem}
\setlist{nolistsep}
\usepackage{indentfirst}
\usepackage{authblk}
\usepackage{epstopdf}
\usepackage[left=0.75in, right=0.75in, top=1in, bottom=1in]{geometry}

\date{\today}
\date{ }

\numberwithin{equation}{section} \numberwithin{figure}{section}
\numberwithin{table}{section}

\theoremstyle{plain}
\newtheorem{scheme}{Scheme}[section]
\newtheorem{thm}{Theorem}[section]

\theoremstyle{remark}
\newtheorem{rem}{Remark}[section]
\newtheorem{defi}{Definition}[section]

\newcommand{\bu}{{\mathbf u}}

\newcommand{\beql}{\begin{equation*}}
\newcommand{\eeql}{\end{equation*}}

\newcommand{\bse}{\begin{subequations}}
\newcommand{\ese}{\end{subequations}}
\def\benl{\begin{eqnarray*}}
\def\eenl{\end{eqnarray*}}

\def\bA{\mathbf{A}}

\def\cF{\mathcal{F}}
\def\cA{\mathcal{A}}
\def\cB{\mathcal{B}}

\def\bu{\mathbf{u}}

\def\bx{\mathbf{x}}

\def\bM{\mathbf{M}}

\def\bu{\mathbf{u}}

\def\cL{\mathcal{L}}

\def\cM{\mathcal{M}}
\def\cG{\mathcal{G}}

\newcommand{\ben}{\begin{eqnarray}}
\newcommand{\een}{\end{eqnarray}}
\newcommand{\beq}{\begin{equation}}
\newcommand{\eeq}{\end{equation}}
\newcommand{\bea}{\begin{array}}
\newcommand{\eea}{\end{array}}
\newcommand{\bef}{\begin{figure}}
\newcommand{\eef}{\end{figure}}

\begin{document}
\title{Arbitrarily High-order Linear Schemes for Gradient Flow Models}
\author{Yuezheng Gong \thanks{College of Science, Nanjing University of Aeronautics and Astronautics, Nanjing 210016, China; Email: gongyuezheng@nuaa.edu.cn},
Jia Zhao \thanks{Department of Mathematics \& Statistics, Utah State University, Logan, UT 84322, USA; Email: jia.zhao@usu.edu}
and Qi Wang \thanks{Corresponding author. Beijing Computational Science Research Center, Beijing 100193, China; Department of Mathematics, University of South Carolina, Columbia, SC 29208, USA Email: qwang@math.sc.edu}}

\maketitle
\begin{abstract}
We present a paradigm for developing arbitrarily high order, linear, unconditionally energy stable numerical algorithms for gradient flow models. We apply the energy quadratization (EQ) technique to reformulate the general gradient flow model into an equivalent gradient flow model with a quadratic free energy and a modified mobility. Given solutions up to $t_n=n \Delta t$ with $\Delta t$ the time step size, we linearize the EQ-reformulated gradient flow model in $(t_n, t_{n+1}]$ by extrapolation. Then we employ an algebraically stable Runge-Kutta method to discretize the linearized model in $(t_n, t_{n+1}]$. Then we use the Fourier pseudo-spectral method for the spatial discretization to match the order of accuracy in time. The resulting fully discrete scheme is linear, unconditionally energy stable, uniquely solvable, and may reach arbitrarily high order. Furthermore, we present a family of linear schemes based on prediction-correction methods to complement the new linear schemes. Some benchmark numerical examples are given to demonstrate the accuracy and efficiency of the schemes.

\end{abstract}

\section{Introduction}\label{sect:Intro}

For many import phenomena in physics, life science, and engineering, the processes are driven by minimizing the free energy or maximizing entropy, i.e., dissipative dynamics. Gradient flow models are usually used to model these phenomena. The generic form of a gradient flow model is given by
\beq \label{eq:Model}
\frac{\partial}{\partial t}\Phi = \cG \frac{\delta F}{\delta \Phi},
\eeq
with proper boundary conditions. Here $\Phi$ is the thermodynamical variable, $ F$ is the free energy for the isothermal system (or entropy for the nonisothermal system),  and $\cG$ is the mobility operator/matrix.
The gradient flow model is thermodynamically consistent if it yields a positive entropy production or negative energy dissipation rate. The classical Allen-Cahn equation \cite{AlCa77}  and Cahn-Hilliard equation \cite{Cahn&H1958} are two examples of gradient flow models \eqref{eq:Model}. Other gradient flow models include the molecular beam epitaxy model \cite{Clar1987}, the phase-field crystal model \cite{PhaseFieldCrystal}, the thermodynamically consistent dendritic growth model \cite{McFadden1993}, the surfactant model \cite{Teng&Chern&LaiDSDC2012},  the diblock copolymer model \cite{Yang-EQ-diblock} etc.

Most gradient flow models are nonlinear, so that their analytical solutions are intractable. Hence, designing accurate, efficient, and stable algorithms to solve them becomes essential \cite{Eyre1998,Furihata2011,WiseSINUMA2009,Shen10_1,Zhao2018EQreview,Wang&Wang&WiseDCDS2010,Zhang&QiaoJCP2013-Adaptive-PFC}. A numerical scheme that preserves the energy dissipation property is known as an energy stable scheme \cite{Eyre1998}. It has been shown that schemes that are not energy stable could lead to instability or oscillatory solutions \cite{Furihata2011}. This is because non-energy-stable schemes may introduce truncation errors that destroy the physical law numerically. Thus, developing energy stable numerical algorithms is necessary for accurately resolving the dynamics of gradient flow models \cite{Guillen&Tierra2014,LaiEAJAM2013,Guo&Lin&Lowengrub&Wise2017,SAV-1,Zhao2018EQreview}.

Over the years, the development of numerical algorithms has been done primarily on a specific gradient flow model, exploiting its specific mathematical properties and structures. The noticeable ones include the Allen-Cahn and Cahn-Hilliard equation \cite{Eyre1998,Bar2001,Ye03,Kim2004,Wise&Kim&LowengrubJCP2007,Shen10_1,Guillen&Tierra2014,DongSSAV} as well as the molecular beam epitaxy model \cite{Qiao&Zhang&TangSISC2011,YZW2017,Time-Adaptive-AML-2020,ChenJSC2012,QiaoMC2015,Feng&Wang&Wise&Zhang2017}. As a result,  most numerical algorithms developed for a specific gradient flow model can hardly be applied to another gradient flow model with a different free energy and mobility. The status-quo did not change until  the energy quaratization (EQ) approach was introduced a few years ago \cite{Yang-EQ-homopolymer,Zhao2018EQreview}, which turns out to be  rather general so that it can be readily applied to general  gradient flow models with no restrictions on the specific form of the mobility and free energy \cite{Yang-EQ-bending,Gong&Zhao&WangACM} (so long as it is bounded below, which is usually the case). Based on the idea of EQ, the scalar auxiliary variable (SAV) approach was introduced later \cite{Shen2019SAVreview}, where each time step only linear systems with constant coefficients need to be solved.  Several other extensions of EQ and SAV approaches have been explored further. For instance, regularization terms \cite{Chen&Zhao&Yang2018} and stabilization terms \cite{SIEQ} are added, and the modified energy quadratization technique \cite{EQ_Li,SAV_Xu} are introduced to improve the EQ methodology.  We notice that most existing schemes related to the EQ methodology are up to 2nd order accurate in time. Some higher-order ETD schemes have been introduced to solve the MBE model and Cahn-Hilliard models recently, but their theoretical proofs of energy stability are still missing.   Shen et al. \cite{SAV-1} remarked on higher-order BDF schemes using the SAV approach, but no rigorous theoretical proofs are available.

Recently,  Gong et al. \cite{Gong&Zhao&WangSISC3,Gong&Zhao&Wang-SAV,GongEnergy} introduced the arbitrarily high-order schemes for solving gradient flow models by combining the methodology of energy quadratization with quadratic invariant Runge-Kutta (RK) methods.  This seminal idea sheds light on solving gradient flow models with an arbitrarily high order of accuracy. We note that the proposed high order schemes are fully nonlinear so that the solution existence and uniqueness are not guaranteed when the time step is large. Moreover, the implementation can be complicated compared to linear schemes. Often, it requires iterations at each time step, which adds up to the computational cost.

In this paper, we will address these issues by proposing a new paradigm for developing arbitrarily high order schemes that are unconditionally energy stable and linear. As a significant advance over our previous work, the newly proposed paradigm would result in linear schemes, while preserving unconditional energy stability. These newly proposed schemes will bring in significant improvement in numerical implementation and reduction in computational cost. First of all, as the schemes are all linear,  only linear systems need to be solved at each time step. Therefore, they are easy to implement and computationally efficient. In addition, the existence and uniqueness of the numerical solutions can be guaranteed, which is actually independent of the time steps. In other words, larger time steps can be readily applied so long as the accuracy requirement is met. This is warranted by the existence of a unique solution and unconditionally energy stable.  Equipped with the benefits of EQ method and RK method, the newly proposed schemes apply to general gradient flow models. For illustration purposes, we solve the Cahn-Hilliard model and the molecular beam epitaxy model to demonstrate the effectiveness of some selected new schemes. To show their accuracy and efficiency, we also compare two proposed schemes with the $2$nd order convex splitting schemes. 

The rest of this paper is organized as follows. In Section 2, we briefly introduce the general gradient flow model and use the EQ method to reformulate it into an equivalent form. In Section 3, the arbitrarily high-order linear energy stable schemes are introduced, and their energy stability and uniquely solvability are discussed.  In Section 4, several numerical examples are shown to illustrate the power of our proposed arbitrarily high order schemes. In the end, some concluding remarks are given.

\section{Gradient Flow Models and Their EQ Reformulation}\label{sec:Model}

In this section, we present the general gradient flow model firstly and then apply the energy quadratization technique to reformulate the model into an equivalent gradient flow form with a quadratic energy functional, a modified mobility matrix and the corresponding energy dissipation law, which is called the EQ reformulated model. The EQ reformulation for this class of gradient flow models provides an elegant platform for developing arbitrarily high-order unconditionally energy stable schemes \cite{Gong&Zhao&WangSISC3,GongEnergy}.

\subsection{Gradient flow models}

Mathematically, the form of a  general gradient flow model is given by \cite{Zhao2018EQreview,Shen2019SAVreview}
\beq \label{eq:gradient-flow}
\frac{\partial}{\partial t}\Phi = \cG \frac{\delta F}{\delta \Phi},
\eeq
where $\Phi=(\phi_1,\cdots,\phi_d)^T$ is the state variable  vector, $\cG$ is the $d \times d$ mobility matrix operator which can depend on $\Phi$, $F$ is the  free energy, and $\frac{\delta F}{\delta \Phi}$ is the variational derivative of the free energy functional with respect to the state variable, known as the chemical potential. The triple $(\Phi,\cG,F)$ uniquely defines the gradient flow model. For \eqref{eq:gradient-flow} to be thermodynamically consistent, the time rate of change of the free energy must be non-increasing:
\beq\label{initialEDL}
\frac{d F}{d t} = \left( \frac{\delta F}{\delta \Phi}, \frac{\partial\Phi}{\partial t} \right) = \left( \frac{\delta F}{\delta \Phi}, \cG \frac{\delta F}{\delta \Phi} \right) \leq 0,
\eeq
where the inner product is defined by $({\bf f}, {\bf g}) = \sum\limits_{i=1}^d \int_\Omega f_i g_i d\bx$, $\forall \mathbf{f}, \mathbf{g} \in \big(L^2(\Omega)\big)^d$, which requires $\cG$ to be  negative semi-definite. The $L^2$ norm is defined as $\| \bf f\|_2 = \sqrt{({\bf f}, {\bf g})}$. Note that the energy dissipation law \eqref{initialEDL} holds only for suitable boundary conditions. Such boundary conditions include periodic boundary conditions and the other boundary conditions that make the boundary integrals resulted from the integration by parts vanish in the calculation of variational derivatives. In this paper, we limit our study to these boundary conditions.

\subsection{Model reformulation using the EQ approach}

We reformulate  the gradient flow model \eqref{eq:gradient-flow} by transforming the free energy into a quadratic form using nonlinear transformations. For the purpose  of illustration, we assume the  free energy is given by  the following
\beq\label{initial-energy}
F = \frac{1}{2}(\Phi, \cL \Phi ) + \big( f (\Phi,\nabla \Phi) ,1\big),
\eeq
where $\cL$ is a linear, self-adjoint, positive semi-definite operator (independent of $\Phi$), and $f$ is the bulk part of the free energy density, which has a lower bound. Then the  free energy $F$ can be rewritten into a quadratic form
\beq\label{EQ-energy}
\cF = \frac{1}{2}(\Phi, \cL \Phi) + \|q\|^2 - C |\Omega|,
\eeq
by introducing an auxiliary variable $q = \sqrt{f (\Phi,\nabla \Phi) +C}$, where $C$ is a positive constant large enough to make $q$ real-valued for all $\Phi$.

We denote $g[\Phi ] = \sqrt{f (\Phi,\nabla \Phi) +C}$. Then model \eqref{eq:gradient-flow} can be reformulated into the following equivalent system
\beq \label{eq:gradient-flow-EQ-scalar}
\left\{
\bea{l}
\Phi_t = \cG \bigg(\cL \Phi + 2q \frac{\partial g}{\partial \Phi} - \nabla\cdot \Big( 2q \frac{\partial g}{\partial \nabla\Phi} \Big) \bigg), \\
q_t = \frac{\partial g}{\partial \Phi} \cdot \Phi_t + \frac{\partial g}{\partial \nabla\Phi} \cdot \nabla\Phi_t,
\eea
\right.
\eeq
with initial conditions
\beq
\Phi(\bx,0)=\Phi_0(\bx), \quad q(\bx, 0)=\sqrt{f \big(\Phi_0(\bx),\nabla \Phi_0(\bx)\big) +C}.
\eeq
It is readily to prove that the reformulated system \eqref{eq:gradient-flow-EQ-scalar} preserves the following energy dissipation law
\beq
\frac{d \cF}{d t} = \Bigg(\cL \Phi + 2q \frac{\partial g}{\partial \Phi} - \nabla\cdot \Big( 2q \frac{\partial g}{\partial \nabla\Phi} \Big),  \cG \bigg(\cL \Phi + 2q \frac{\partial g}{\partial \Phi} - \nabla\cdot \Big( 2q \frac{\partial g}{\partial \nabla\Phi} \Big) \bigg)\Bigg) \leq 0. \label{EDP-EQ}
\eeq

We introduce
\beq
\bu=(\Phi, q)^T
\eeq
and recast  system \eqref{eq:gradient-flow-EQ-scalar} into a compact gradient flow form
\beq
\bu_t=\cM \frac{\delta \cF}{\delta \bu},
\eeq
with a modified mobility operator
\beq
\cM=
\left (
\bea{c}
1\\
\frac{\partial g}{\partial \Phi}+\frac{\partial g}{\partial \nabla \Phi}\cdot \nabla\eea
\right) \cG\left (1, \frac{\partial g}{\partial \Phi}-\nabla \cdot \frac{\partial g}{\partial \nabla \Phi}\right).
\eeq
The energy dissipation law given in \eqref{EDP-EQ} is recast into
\ben
\frac{d \cF}{d t} = \Big(\frac{\delta \cF}{\delta \bu}, \cM \frac{\delta \cF}{\delta \bu} \Big)  \leq 0.
\een
Since the EQ-reformulated form in \eqref{eq:gradient-flow-EQ-scalar} has a quadratic free energy, we next discuss how to devise linear high-order energy stable schemes for it.

\section{High-order linear energy stable schemes}\label{sec:Method}

In this section, we first derive a high-precision linear gradient-flow system to approximate EQ-reformulated model  \eqref{eq:gradient-flow-EQ-scalar} up to $t_n=n\Delta t$, where $\Delta t$ is the time step. In particular, the corresponding energy dissipation law is inherited. Then the algebraically stable RK method \cite{KevinStability} is applied to the resulting linear gradient-flow system to develop a class of linear semi-discrete schemes in time. We name the schemes linear energy quadratizatized Runge-Kutta (LEQRK) methods. In order to improve accuracy and stability, a prediction-correction technique is proposed for the LEQRK schemes, leading to the LEQRK-PC methods. These new algorithms are linear, unconditionally energy stable, and can be devised at any desired order in time.

\subsection{LEQRK schemes}
Assuming numerical solutions of $\Phi$ up to  $t\leq t_n$ have been obtained, we then solve system \eqref{eq:gradient-flow-EQ-scalar} in $t \in (t_n,t_{n+1}]$ approximately. We utilize  the numerical solutions of $\Phi$ at $t\leq t_n$ to obtain its interpolating polynomial approximation denoted by $\Phi_N(t)$. Then we approximate model \eqref{eq:gradient-flow-EQ-scalar} in $(t_n,t_{n+1}]$ using the following linear, variable coefficient gradient flow system
\beq \label{linearPDE}
\left\{
\bea{l}
\Phi_t = \cG \bigg(\cL \Phi + 2q \frac{\partial g_N}{\partial \Phi} - \nabla\cdot \Big( 2q \frac{\partial g_N}{\partial \nabla\Phi} \Big) \bigg), \\
q_t = \frac{\partial g_N}{\partial \Phi} \cdot \Phi_t + \frac{\partial g_N}{\partial \nabla\Phi} \cdot \nabla\Phi_t,
\eea
\right.
\eeq
where $\frac{\partial g_N}{\partial \Phi} = \frac{\partial g}{\partial \Phi}[\Phi_N(t)]$ and $\frac{\partial g_N}{\partial \nabla\Phi} = \frac{\partial g}{\partial \nabla\Phi}[\Phi_N(t)]$ are independent of $\Phi$. The linear gradient flow system \eqref{linearPDE} satisfies the following energy dissipation law
\beq
\frac{d \cF}{d t} = \Bigg(\cL \Phi + 2q \frac{\partial g_N}{\partial \Phi} - \nabla\cdot \Big( 2q \frac{\partial g_N}{\partial \nabla\Phi} \Big),  \cG \bigg(\cL \Phi + 2q \frac{\partial g_N}{\partial \Phi} - \nabla\cdot \Big( 2q \frac{\partial g_N}{\partial \nabla\Phi} \Big) \bigg)\Bigg) \leq 0.
\eeq

Applying a $s$-stage RK method for the linear system \eqref{linearPDE}, we obtain the following LEQRK scheme.

\begin{scheme}[$s$-stage LEQRK Scheme] \label{scheme:LEQRK}
Let $b_i$, $a_{ij}$ ($i,j = 1,\cdots,s$) be real numbers and let $c_i = \sum\limits_{j=1}^s a_{ij}$.
For given $(\Phi^n, q^n)$  and $\Phi_N(t_n+c_i \Delta t), \forall i$, the following intermediate values are first calculated by
\beq \label{RK_stage}
\left\{
\bea{l}
\Phi_i^n =  \Phi^n +  \Delta t \sum\limits_{j=1}^s a_{ij} k_j^n,\\
Q_i^n = q^n +  \Delta t \sum\limits_{j=1}^s a_{ij} l_j^n,\\
k_i^n = \cG \bigg(\cL \Phi_i^n + 2Q_i^n  \Big(\frac{\partial g}{\partial \Phi}\Big)_i^{n,*}  - \nabla\cdot \Big( 2Q_i^n \Big(\frac{\partial g}{\partial \nabla\Phi}\Big)_i^{n,*} \Big) \bigg), \\
l_i^n = \Big(\frac{\partial g}{\partial \Phi}\Big)_i^{n,*}  \cdot k_i^n + \Big(\frac{\partial g}{\partial \nabla\Phi}\Big)_i^{n,*} \cdot \nabla k_i^n,
\eea
\right.
\quad i = 1,\cdots,s,
\eeq
where $\Big(\frac{\partial g}{\partial \Phi}\Big)_i^{n,*} =  \frac{\partial g}{\partial \Phi}[\Phi_N(t_n+c_i \Delta t)]$ and $\Big(\frac{\partial g}{\partial \nabla\Phi}\Big)_i^{n,*} = \frac{\partial g}{\partial \nabla\Phi}[\Phi_N(t_n+c_i \Delta t)]$.
Then $(\Phi^{n+1}, q^{n+1})$ is updated via
\ben
&& \Phi^{n+1} = \Phi^n +  \Delta t \sum\limits_{i=1}^s b_i k_i^n ,\label{update_Phi}\\
&& q^{n+1} = q^n +  \Delta t \sum\limits_{i=1}^s b_i l_i^n.
\een
\end{scheme}

\begin{defi} [Algebraically Stable RK Method \cite{KevinStability}]
Denote a symmetric matrix $\mathbf{M}$ with the elements
$
\mathbf{M}_{ij} = b_i a_{ij} + b_j a_{ji} - b_i b_j.
$
A RK method is said to be algebraically stable if its RK coefficients satisfy stability condition \beq\label{stability-condition}
b_i\geq 0,~\forall i=1,2,\cdots,s, \quad \textrm{and~} \mathbf{M} \textrm{~is~positive~semi-definite}.
\eeq
\end{defi}

Next, we show that the algebraically stable LEQRK scheme is unconditionally energy stable.

\begin{thm} \label{thm:Energy-Stability}
The LEQRK scheme with their RK coefficients satisfying stability condition \eqref{stability-condition} is unconditionally energy stable, i.e., it satisfies
\beq
\cF^{n+1}\leq \cF^n,
\eeq
where $\cF^n = \frac{1}{2}(\Phi^n, \cL \Phi^n) + \|q^n\|^2 - C |\Omega|.$
\end{thm}

\begin{proof}
Denoting $\Phi^{n+1} = \Phi^n + \Delta t \sum\limits_{i=1}^s b_i k_i^n$ and noticing that operator $\cL$ is linear and self-adjoint, we have
\beq\label{energy-delta-phi-1}
\frac{1}{2}(\Phi^{n+1}, \cL \Phi^{n+1}) -\frac{1}{2}(\Phi^n, \cL \Phi^n) = \Delta t\sum\limits_{i=1}^s b_i(k_i^n,\cL \Phi^{n}) +\frac{\Delta t ^2}{2}\sum\limits_{i,j=1}^s b_i b_j (k_i^n,\cL k_j^n).
\eeq
Applying $\Phi^{n} = \Phi_i^n - \Delta t\sum\limits_{j=1}^s a_{ij} k_j^n$ to the right of \eqref{energy-delta-phi-1}, we deduce
\beq\label{mid1}
\frac{1}{2}(\Phi^{n+1}, \cL \Phi^{n+1}) -\frac{1}{2}(\Phi^n, \cL \Phi^n) = \Delta t\sum\limits_{i=1}^s b_i(k_i^n,\cL \Phi_i^n) - \frac{\Delta t ^2}{2}\sum\limits_{i,j=1}^s \mathbf{M}_{ij} (k_i^n,\cL k_j^n),
\eeq
where $\sum\limits_{i,j=1}^s b_i a_{ij} (k_i^n,\cL k_j^n) = \sum\limits_{i,j=1}^s b_j a_{j i} (k_i^n,\cL k_j^n)$ and $\mathbf{M}_{ij} = b_i a_{i j} + b_j a_{j i} - b_i b_j$ are used.  Note that $\cL$ can be denoted as $\cL = \cA^{*} \cA,$
where $\cA$ is a linear operator and $\cA^{*}$ is the adjoint operator of $\cA$. Since $\mathbf{M}$ is positive semi-definite, we have
\beq\label{mid2}
\sum\limits_{i,j=1}^s \mathbf{M}_{ij} (k_i^n,\cL k_j^n) = \sum\limits_{i,j=1}^s \mathbf{M}_{ij} (\cA k_i^n,\cA k_j^n)\geq 0.
\eeq
Combining eqs. \eqref{mid1} and \eqref{mid2} leads to
\beq\label{energy-delta-phi}
\frac{1}{2}(\Phi^{n+1}, \cL \Phi^{n+1}) -\frac{1}{2}(\Phi^n, \cL \Phi^n) \leq \Delta t\sum\limits_{i=1}^s b_i(k_i^n,\cL \Phi_i^n).
\eeq
Similarly, we have
\beq\label{energy-delta-q}
\|q^{n+1}\|^2 - \|q^n\|^2 \leq 2\Delta t\sum\limits_{i=1}^s b_i (l_i^n, Q_i^n).
\eeq
Adding \eqref{energy-delta-phi} and \eqref{energy-delta-q} and noticing that $l_i^n = \Big(\frac{\partial g}{\partial \Phi}\Big)_i^{n,*}  \cdot k_i^n + \Big(\frac{\partial g}{\partial \nabla\Phi}\Big)_i^{n,*} \cdot \nabla k_i^n$, we obtain
\beq\label{delta-energy}
\cF^{n+1} - \cF^n \leq \Delta t \sum\limits_{i=1}^s b_i \bigg( \cL \Phi_i^n + 2Q_i^n  \Big(\frac{\partial g}{\partial \Phi}\Big)_i^{n,*}  - \nabla\cdot \Big( 2Q_i^n \Big(\frac{\partial g}{\partial \nabla\Phi}\Big)_i^{n,*} \Big), k_i^n \bigg).
\eeq
Substituting $k_i^n = \cG \bigg(\cL \Phi_i^n + 2Q_i^n  \Big(\frac{\partial g}{\partial \Phi}\Big)_i^{n,*}  - \nabla\cdot \Big( 2Q_i^n \Big(\frac{\partial g}{\partial \nabla\Phi}\Big)_i^{n,*} \Big) \bigg)$ in \eqref{delta-energy} and noticing the negative semi-definite property of $\cG$ and $b_i\geq 0, \forall i$, we arrive at
$
\cF^{n+1} - \cF^n \leq 0.
$
This completes the proof.
\end{proof}

\begin{rem}
Note that the Gauss method is a special kind of algebraically stable RK method, whose RK coefficients satisfy $\bM = {\bf 0}$. Thus the Gauss method preserves the discrete energy dissipation law
\beq
\cF^{n+1} - \cF^n = \Delta t \sum\limits_{i=1}^s b_i \bigg( \cL \Phi_i^n + 2Q_i^n  \Big(\frac{\partial g}{\partial \Phi}\Big)_i^{n,*}  - \nabla\cdot \Big( 2Q_i^n \Big(\frac{\partial g}{\partial \nabla\Phi}\Big)_i^{n,*} \Big), k_i^n \bigg) \leq 0.
\eeq
\end{rem}
\begin{rem}
After appropriate spatial discretization that satisfies the discrete integration-by-parts formula (see \cite{Gong&Zhao&WangACM,Gong2018Second} for details), the algebraically stable LEQRK scheme naturally leads to a fully discrete energy stable scheme. In this paper, we employ the Fourier pseudo-spectral method for spatial discretization. We omit the details here due to space limitation. Interested readers are referred to our ealier work \cite{Gong&Zhao&WangACM,Gong&Zhao&WangSISC3} for details.
\end{rem}

Next, we discuss the solvability of the resulting fully discrete scheme.

\begin{thm} \label{thm:Solvability-GRK}
If RK coefficient matrix $\bA = (a_{ij})$ is positive semi-definite and  mobility operator $\cG$ satisfies $\cG = -\cB^*\cB$,  the fully discrete scheme derived by applying the Fourier pseudo-spectral method to \textbf{Scheme \ref{scheme:LEQRK}} is uniquely solvable.
\end{thm}

\begin{proof}
Without loss of generality, we still use the notations $\cG$, $\cL$ and $\nabla$ to denote the corresponding discrete operators in the fully discrete scheme. We consider the homogeneous linear equation system of \eqref{RK_stage}
\beq \label{homogeneous_RK_stage}
\left\{
\bea{l}
\Phi_i^n =  \Delta t \sum\limits_{j=1}^s a_{ij} k_j^n,\\
Q_i^n = \Delta t \sum\limits_{j=1}^s a_{ij} l_j^n,\\
k_i^n = \cG \bigg(\cL \Phi_i^n + 2Q_i^n  \Big(\frac{\partial g}{\partial \Phi}\Big)_i^{n,*}  - \nabla\cdot \Big( 2Q_i^n \Big(\frac{\partial g}{\partial \nabla\Phi}\Big)_i^{n,*} \Big) \bigg), \\
l_i^n = \Big(\frac{\partial g}{\partial \Phi}\Big)_i^{n,*}  \cdot k_i^n + \Big(\frac{\partial g}{\partial \nabla\Phi}\Big)_i^{n,*} \cdot \nabla k_i^n,
\eea
\right.
\quad i = 1,\cdots,s,
\eeq
where $\Phi_i^n, Q_i^n, k_i^n, l_i^n$ are unknown. To prove unique solvability of the fully discrete scheme, we need to prove that the homogeneous linear equation system \eqref{homogeneous_RK_stage} admits only a zero solution.

Computing the discrete inner product of the third equation in \eqref{homogeneous_RK_stage} with $\cL \Phi_i^n + 2Q_i^n  \Big(\frac{\partial g}{\partial \Phi}\Big)_i^{n,*}  - \nabla\cdot \Big( 2Q_i^n \Big(\frac{\partial g}{\partial \nabla\Phi}\Big)_i^{n,*} \Big)$ and sum over $i$, we deduce from \eqref{homogeneous_RK_stage}
\beql
\Delta t \sum\limits_{i,j=1}^s a_{ij} (\cA k_i^n,\cA k_j^n) + 2\Delta t \sum\limits_{i,j=1}^s a_{ij} (l_i^n,l_j^n) + \sum\limits_{i=1}^s\Bigg\|\cB \bigg(\cL \Phi_i^n + 2Q_i^n  \Big(\frac{\partial g}{\partial \Phi}\Big)_i^{n,*}  - \nabla\cdot \Big( 2Q_i^n \Big(\frac{\partial g}{\partial \nabla\Phi}\Big)_i^{n,*} \Big) \bigg)\Bigg\|^2 = 0,
\eeql
where $\cG = -\cB^*\cB$ and $\cL = \cA^*\cA$ are used. Since $\bA = (a_{ij})$ is positive semi-definite, which implies that the first two terms of the above equation are nonnegative, thus we have
\beq
\cB \bigg(\cL \Phi_i^n + 2Q_i^n  \Big(\frac{\partial g}{\partial \Phi}\Big)_i^{n,*}  - \nabla\cdot \Big( 2Q_i^n \Big(\frac{\partial g}{\partial \nabla\Phi}\Big)_i^{n,*} \Big) \bigg) = 0, \quad \forall i,
\eeq
which leads to
\beq
\cG \bigg(\cL \Phi_i^n + 2Q_i^n  \Big(\frac{\partial g}{\partial \Phi}\Big)_i^{n,*}  - \nabla\cdot \Big( 2Q_i^n \Big(\frac{\partial g}{\partial \nabla\Phi}\Big)_i^{n,*} \Big) \bigg) = 0, \quad \forall i.
\eeq
Therefore, according to \eqref{homogeneous_RK_stage}, we arrive at
\beq
k_i^n = 0, \quad l_i^n = 0, \quad \Phi_i^n = 0, \quad Q_i^n = 0, \quad \forall i.
\eeq
This completes the proof.
\end{proof}

\begin{thm} \label{thm:Solvability-DIRK}
If diagnally implicit RK coefficients satisfy $a_{ii}>0$, then the fully discrete scheme derived by applying the Fourier pseudo-spectral method to \textbf{Scheme \ref{scheme:LEQRK}} is uniquely solvable.
\end{thm}

\begin{proof}
For the diagnally implicit RK (DIRK) scheme, we solve $\Phi_i^n, Q_i^n, k_i^n, l_i^n$ in turn from $i=1$ to $s$. Therefore, we here consider the following homogeneous linear equation system
\beq \label{homogeneous_RK_stage-DIRK}
\left\{
\bea{l}
\Phi_i^n =  \Delta t a_{ii} k_i^n,\\
Q_i^n = \Delta t a_{ii} l_i^n,\\
k_i^n = \cG \bigg(\cL \Phi_i^n + 2Q_i^n  \Big(\frac{\partial g}{\partial \Phi}\Big)_i^{n,*}  - \nabla\cdot \Big( 2Q_i^n \Big(\frac{\partial g}{\partial \nabla\Phi}\Big)_i^{n,*} \Big) \bigg), \\
l_i^n = \Big(\frac{\partial g}{\partial \Phi}\Big)_i^{n,*}  \cdot k_i^n + \Big(\frac{\partial g}{\partial \nabla\Phi}\Big)_i^{n,*} \cdot \nabla k_i^n,
\eea
\right.
\eeq
where $\Phi_i^n, Q_i^n, k_i^n, l_i^n$ are unknown. To prove unique solvability of the fully discrete scheme, we need to prove that homogeneous linear equation system \eqref{homogeneous_RK_stage-DIRK} admits only a zero solution.

Similar to the proof of \textbf{Theorem \ref{thm:Solvability-GRK}}, we have
\benl
&&\Delta t a_{ii} \|\cA k_i^n\|^2 + 2\Delta t a_{ii} \|l_i^n\|^2\\
&=& \Bigg(\cL \Phi_i^n + 2Q_i^n  \Big(\frac{\partial g}{\partial \Phi}\Big)_i^{n,*}  - \nabla\cdot \Big( 2Q_i^n \Big(\frac{\partial g}{\partial \nabla\Phi}\Big)_i^{n,*} \Big),\cG \bigg(\cL \Phi_i^n + 2Q_i^n  \Big(\frac{\partial g}{\partial \Phi}\Big)_i^{n,*}  - \nabla\cdot \Big( 2Q_i^n \Big(\frac{\partial g}{\partial \nabla\Phi}\Big)_i^{n,*} \Big) \bigg)\Bigg)\\
&\leq& 0,
\eenl
which leads to
\beq\label{eq:mid}
\cA k_i^n = 0, \quad l_i^n = 0.
\eeq
Combining \eqref{homogeneous_RK_stage-DIRK} and \eqref{eq:mid}, we deduce in turn
\beq
l_i^n = 0, \quad Q_i^n = 0, \quad \cL k_i^n = 0, \quad \cL \Phi_i^n = 0, \quad k_i^n = 0, \quad \Phi_i^n = 0.
\eeq
This completes the proof.
\end{proof}

\begin{rem}
In this paper, we give examples in two 4th-order algebraically stable RK methods, i.e., Gauss4th and DIRK4th given by the Butcher tables, respectively, 
\[
\begin{array}
{c|cc}
\frac{1}{2}-\frac{\sqrt{3}}{6} & \frac{1}{4} & \frac{1}{4}-\frac{\sqrt{3}}{6}  \\
\frac{1}{2}+\frac{\sqrt{3}}{6} & \frac{1}{4}+\frac{\sqrt{3}}{6} & \frac{1}{4} \\
\hline
& \frac{1}{2} &  \frac{1}{2} \\
\end{array},
\qquad\qquad
\begin{array}
{c|ccc}
\sigma &\sigma & 0 & 0  \\
\frac{1}{2} & \frac{1}{2}-\sigma & \sigma & 0 \\
1-\sigma & 2\sigma & 1-4\sigma & \sigma \\
\hline
& \mu &  1-2\mu & \mu \\
\end{array},
\]
with $\sigma = \cos(\pi/18)/\sqrt{3}+1/2$, $\mu = 1/\big(6(2\sigma-1)^2\big)$.
We note that Gauss4th and DIRK4th satisfy the conditions in \textbf{Theorem \ref{thm:Solvability-GRK}} and \textbf{Theorem \ref{thm:Solvability-DIRK}}, respectively. Therefore, after an appropriate spatial discretization, the LEQRK schemes equipped with Gauss4th or DIRK4th are uniquely solvable.
\end{rem}

\begin{rem}
Noticing that $\Phi_i^m$ approximates $\Phi(t_m+c_i \Delta t)$, we can choose the time nodes $t_m$, $t_{m}+c_i \Delta t$ ($m<n$) and $t_n$ as the interpolation points  to obtain the interpolation polynomial $\Phi_N(t)$. However, too many interpolation points will cause the interpolation polynomial to be highly oscillating, which may make $\Phi_N(t_n+c_i \Delta t)$ an inaccurate extrapolation for $\Phi(t_n+c_i\Delta t)$. Therefore, we only take $t_{n-1}$, $t_{n-1}+c_i \Delta t, ~\forall i$ and $t_n$ as  the interpolation points in this paper. For example, for the Gauss4th method, we choose the interpolation points $(t_{n-1}, \Phi^{n-1})$, $(t_{n-1} + c_1 \Delta t, \Phi_1^{n-1})$, $(t_{n-1} + c_2 \Delta t, \Phi_2^{n-1})$, $(t_{n}, \Phi^{n})$ and obtain the corresponding interpolation polynomial
\benl
\Phi_N(t_{n-1}+s \Delta t) &=& \frac{(s-c_1)(s-c_2)(s-1)}{-c_1 c_2}\Phi^{n-1} + \frac{s(s-c_2)(s-1)}{c_1(c_1-c_2)(c_1-1)}\Phi_1^{n-1}  \\
&&+ \frac{s(s-c_1)(s-1)}{c_2 (c_2-c_1)(c_2-1)}\Phi_2^{n-1}  + \frac{s(s-c_1)(s-c_2)}{(1-c_1)(1-c_2)}\Phi^{n},
\eenl
where $c_1 = 1/2-\sqrt{3}/6$ and $c_2 = 1/2+\sqrt{3}/6$. Thus we have
\begin{align}
&\Phi_N(t_{n}+c_1 \Delta t) = (2\sqrt{3}-4)\Phi^{n-1} + (7\sqrt{3}-11)\Phi_1^{n-1}  + (6-5\sqrt{3})\Phi_2^{n-1}  + (10-4\sqrt{3})\Phi^{n},\label{Gauss_RK_s1}\\
&\Phi_N(t_{n}+c_2 \Delta t) = -(2\sqrt{3}+4)\Phi^{n-1}  +(6+5\sqrt{3})\Phi_1^{n-1}  - (7\sqrt{3}+11)\Phi_2^{n-1}  + (10+4\sqrt{3})\Phi^{n}.\label{Gauss_RK_s2}
\end{align}
Replacing $n$ with $n-1$ in the first equation of \eqref{RK_stage} and \eqref{update_Phi}, then we deduce
\beq\label{Phi_n}
\Phi^n = \Phi^{n-1}-\sqrt{3}\Phi_1^{n-1} + \sqrt{3}\Phi_2^{n-1}.
\eeq
According to \eqref{Gauss_RK_s1}-\eqref{Phi_n}, we obtain
\ben
&&\Phi_N(t_{n}+c_1 \Delta t) = (6-2\sqrt{3})\Phi^{n-1} + (1-3\sqrt{3})\Phi_1^{n-1}  + (5\sqrt{3}-6)\Phi_2^{n-1},\label{Gauss_s1}\\
&&\Phi_N(t_{n}+c_2 \Delta t) = (6+2\sqrt{3})\Phi^{n-1}  - (5\sqrt{3}+6)\Phi_1^{n-1}  + (1+3\sqrt{3})\Phi_2^{n-1}.\label{Gauss_s2}
\een
Note that if we take $t_{n-1}$, $t_{n-1}+c_i \Delta t ~(i=1,2)$ as  the interpolation points,  we can also derive \eqref{Gauss_s1}-\eqref{Gauss_s2}, which implies that the LEQRK scheme induced by the Gauss4th method and the interpolations \eqref{Gauss_RK_s1}-\eqref{Gauss_RK_s2} or \eqref{Gauss_s1}-\eqref{Gauss_s2}  may achieve third order accuracy.
\end{rem}

\subsection{LEQRK-PC schemes}
To improve the accuracy as well as stability of \textbf{Scheme \ref{scheme:LEQRK}}, we  propose a prediction-correction scheme motivated by the works in \cite{ConvexSplittingPredictor,CiCP-24-635,Gong&Zhao&WangSISC3}. Employing the prediction-correction strategy to \textbf{Scheme \ref{scheme:LEQRK}}, we obtain the following prediction-correction method:

\begin{scheme}[$s$-stage LEQRK-PC Scheme] \label{scheme:LEQRK-PC}
Let $b_i$, $a_{ij}$ ($i,j = 1,\cdots,s$) be real numbers and let $c_i = \sum\limits_{j=1}^s a_{ij}$.
For given $(\Phi^n, q^n)$ and $\Phi_N(t_n+c_i \Delta t), Q_N(t_n+c_i \Delta t),  \forall i$, the following intermediate values are first calculated by the following prediction-correction strategy
\begin{enumerate}
\item  Prediction: we set $\Phi_i^{n,0} = \Phi_N(t_n+c_i \Delta t),$ $Q_i^{n,0} = Q_N(t_n+c_i \Delta t)$. Let $M>0$ be a given integer. For $m=0$ to $M-1$, we compute $\Phi_i^{n,m+1}$, $k_i^{n,m+1}$, $l_i^{n,m+1}$, $Q_i^{n,m+1}$ using
\beq \label{pre_RK_stage}
\left\{
\bea{l}
\Phi_i^{n,m+1} =  \Phi^n +  \Delta t \sum\limits_{j=1}^s a_{ij} k_j^{n,m+1},\\
k_i^{n,m+1} = \cG \bigg(\cL \Phi_i^{n,m+1} + 2Q_i^{n,m}  \frac{\partial g}{\partial \Phi}[\Phi_i^{n,m}]  - \nabla\cdot \Big( 2Q_i^{n,m} \frac{\partial g}{\partial \nabla\Phi}[\Phi_i^{n,m}] \Big) \bigg),\\
l_i^{n,m+1} = \frac{\partial g}{\partial \Phi}[\Phi_i^{n,m+1}]  \cdot k_i^{n,m+1} + \frac{\partial g}{\partial \nabla\Phi}[\Phi_i^{n,m+1}] \cdot \nabla k_i^{n,m+1},\\
Q_i^{n,m+1} = q^n +  \Delta t \sum\limits_{j=1}^s a_{ij} l_j^{n,m+1},
\eea
\right.
\quad i = 1,\cdots,s.
\eeq
If  $\max\limits_{i}\|\Phi^{n,m+1}_{i}-\Phi^{n,m}_{i}\|_\infty < TOL$, we stop the iteration and  set $\Phi^{n,*}_{i} = \Phi_i^{n,m+1}$; otherwise, we set $\Phi^{n,*}_{i} = \Phi_i^{n,M}$.

\item  Correction: for the predicted $\Phi^{n,*}_{i}$, we compute the intermediate values $\Phi_i^n, Q_i^n, k_i^n, l_i^n$ via
\beq \label{Cor_RK_stage}
\left\{
\bea{l}
\Phi_i^n =  \Phi^n +  \Delta t \sum\limits_{j=1}^s a_{ij} k_j^n,\\
Q_i^n = q^n +  \Delta t \sum\limits_{j=1}^s a_{ij} l_j^n,\\
k_i^n = \cG \bigg(\cL \Phi_i^n + 2Q_i^n  \frac{\partial g}{\partial \Phi}[\Phi_i^{n,*}]  - \nabla\cdot \Big( 2Q_i^n \frac{\partial g}{\partial \nabla\Phi}[\Phi_i^{n,*}] \Big) \bigg), \\
l_i^n = \frac{\partial g}{\partial \Phi}[\Phi_i^{n,*}]  \cdot k_i^n + \frac{\partial g}{\partial \nabla\Phi}[\Phi_i^{n,*}] \cdot \nabla k_i^n,
\eea
\right.
\quad i = 1,\cdots,s.
\eeq
\end{enumerate}
Then $(\Phi^{n+1}, q^{n+1})$ is updated via
\ben
&& \Phi^{n+1} = \Phi^n +  \Delta t \sum\limits_{i=1}^s b_i k_i^n,\\
&& q^{n+1} = q^n +  \Delta t \sum\limits_{i=1}^s b_i l_i^n.
\een
\end{scheme}

\begin{rem}
Note that $Q_N(t)$ of \textbf{Scheme \ref{scheme:LEQRK-PC}} denotes the interpolation polynomial of $q$. If we take $\Phi_N(t) = \Phi^n, ~Q_N(t) = q^n,$ then \textbf{Scheme \ref{scheme:LEQRK}} reduces to first order while \textbf{Scheme \ref{scheme:LEQRK-PC}} with appropriate predictions can achieve the desired high order. In numerical computations, we will apply \textbf{Scheme \ref{scheme:LEQRK-PC}} to figure out the necessary initial information. In addition, linear system \eqref{pre_RK_stage} is constant coefficient and thus can be readily solved by using the fast Fourier transform (FFT).
\end{rem}

\begin{rem}
If we choose $M = 0,$ then \textbf{Scheme \ref{scheme:LEQRK-PC}} reduces to \textbf{Scheme \ref{scheme:LEQRK}}. If $M$ is large enough, the LEQRK-PC scheme approximates the IEQ-RK scheme proposed in \cite{GongEnergy}. There is no theoretical result on the choice of iteration step $M$. From our numerical experience, several iteration steps $M \leq 5$ would improve the accuracy noticeably.
\end{rem}

\begin{rem}
Similar to \textbf{Scheme \ref{scheme:LEQRK}}, we can also establish energy stability and solvability for the LEQRK-PC scheme, which is omitted here to save space.
\end{rem}

\section{ Numerical Results}\label{sect:Appl}

In the previous sections, we present some high-order linear energy stable schemes for general gradient flow models. In this section, we apply the proposed schemes to two benchmark gradient flow models: the Cahn-Hilliard model for binary fluids and the molecular beam epitaxial (MBE) growth model. For convenience, the LEQRK schemes equipped with Gauss4th and DIRK4th are abbreviated respectively as LEQGRK and LEQDIRK, while their corresponding LEQRK-PC schemes with the prediction iteration $M$ are denoted by LEQGRK-PC-$M$ and LEQDIRK-PC-$M$.  

\subsection{Cahn-Hilliard model}

We consider the Cahn-Hilliard model for immiscible binary fluids given as follows
\beq\label{CHeq}
\phi_t = \lambda \Delta(-\varepsilon^2\Delta \phi + \phi^3-\phi),
\eeq
with the double-well bulk energy
\beq\label{CH-energy}
F = \frac{\varepsilon^2}{2}\|\nabla \phi\|^2 + \frac{1}{4}\|\phi^2-1\|^2,
\eeq
where $\lambda$ is the mobility parameter and $\varepsilon$ controls the interfacial thickness. If we introduce the auxiliary variable $q = \frac{1}{2}(\phi^2-1-\gamma)$, where $\gamma\geq 0$ is a constant, the energy functional \eqref{CH-energy} is rewritten into
\beq
\cF = \frac{1}{2} \Big(\phi, -\varepsilon^2\Delta \phi + \gamma\phi \Big) + \|q\|^2 - \frac{\gamma^2+2\gamma}{4}|\Omega|.
\eeq
Then the Cahn-Hilliard equation \eqref{CHeq} is equivalently transform into the following system
\beq \label{CHeq-EQ}
\left\{
\bea{l}
\phi_t = \lambda \Delta(-\varepsilon^2\Delta \phi + \gamma\phi + 2q \phi), \\
q_t = \phi \phi_t,
\eea
\right.
\eeq
which satisfies the following  energy dissipation law
\beq
\frac{d \cF}{d t} = -\lambda\big\|\nabla(-\varepsilon^2\Delta \phi + \gamma\phi + 2q \phi)\big\|^2\leq 0.
\eeq

Applying the LEQRK-PC scheme to system \eqref{CHeq-EQ}, we obtain

\begin{scheme} \label{scheme:LEQRK-PC-CH}
Let $b_i$, $a_{ij}$ ($i,j = 1,\cdots,s$) be real numbers and $c_i = \sum\limits_{j=1}^s a_{ij}$. For given $(\phi^n, q^n)$ and $\Phi_N(t_n+c_i \Delta t), Q_N(t_n+c_i \Delta t),  \forall i$, the following intermediate values are first calculated by the following prediction-correction strategy.
\begin{enumerate}
\item  Prediction: we set $\Phi_i^{n,0} = \Phi_N(t_n+c_i \Delta t),$ $Q_i^{n,0} = Q_N(t_n+c_i \Delta t)$ and $M>0$ as a given positive integer. For $m=0$ to $M-1$, we compute $\Phi_i^{n,m+1}$, $k_i^{n,m+1}$, $l_i^{n,m+1}$, $Q_i^{n,m+1}$ using
\beq \label{pre_RK_stage-CH}
\left\{
\bea{l}
\Phi_i^{n,m+1} =  \phi^n +  \Delta t \sum\limits_{j=1}^s a_{ij} k_j^{n,m+1},\\
k_i^{n,m+1} =  \lambda \Delta\big(-\varepsilon^2\Delta \Phi_i^{n,m+1} + \gamma\Phi_i^{n,m+1}  + 2Q_i^{n,m} \Phi_i^{n,m} \big),\\
l_i^{n,m+1} = \Phi_i^{n,m+1} k_i^{n,m+1},\\
Q_i^{n,m+1} = q^n +  \Delta t \sum\limits_{j=1}^s a_{ij} l_j^{n,m+1},
\eea
\right.
\quad i = 1,\cdots,s.
\eeq
Given an error tolerance $TOL>0$, if  $\max\limits_{i}\|\Phi^{n,m+1}_{i}-\Phi^{n,m}_{i}\|_\infty < TOL$, we stop the iteration and  set $\Phi^{n,*}_{i} = \Phi_i^{n,m+1}$; otherwise, we set $\Phi^{n,*}_{i} = \Phi_i^{n,M}$.

\item  Correction: for the predicted $\Phi^{n,*}_{i}$, we compute the intermediate values $\Phi_i^n, Q_i^n, k_i^n, l_i^n$ via
\beq \label{Cor_RK_stage-CH}
\left\{
\bea{l}
\Phi_i^n =  \phi^n +  \Delta t \sum\limits_{j=1}^s a_{ij} k_j^n,\\
Q_i^n = q^n +  \Delta t \sum\limits_{j=1}^s a_{ij} l_j^n,\\
k_i^n = \lambda \Delta\big(-\varepsilon^2\Delta \Phi_i^{n} + \gamma\Phi_i^{n}  + 2Q_i^{n} \Phi_i^{n,*} \big), \\
l_i^n = \Phi_i^{n,*} k_i^{n},
\eea
\right.
\quad i = 1,\cdots,s.
\eeq
\end{enumerate}
Then $(\phi^{n+1}, q^{n+1})$ is updated via
\ben
&& \phi^{n+1} = \phi^n +  \Delta t \sum\limits_{i=1}^s b_i k_i^n,\\
&& q^{n+1} = q^n +  \Delta t \sum\limits_{i=1}^s b_i l_i^n.
\een
\end{scheme}

First of all, we present the time mesh refinement tests to show the order of accuracy of the proposed schemes. We consider the domain as $[0 \,\,\, 2\pi]^2$ and choose  model parameter values $\lambda = 0.01$, $\varepsilon = 1$ and $\gamma =1$. Note that the analytical solution for the Cahn-Hilliard equation is usually unknown. To better calculate the errors in time mesh refinement tests, we create an exact solution $\phi (x,y,t) = \sin(x)\sin(y) \cos(t)$, by adding a corresponding forcing term on the right-hand side of the Cahn-Hilliard equation. Then, we solve it in a 2D spatial domain with periodic boundary conditions. The equation is discretized spatially using the Fourier pseudo-spectral method with $128^2$ spatial meshes.

The numerical solution of $\phi$ at $t=1$ is calculated using a set of different numerical schemes with various time steps. Both the $L^2$ and $L^\infty$ errors in the solution are calculated, and the results are summarized in Figure \ref{fig:CH-error}. We observe that, due to the low-order extrapolation, LEQDIRK only reaches 2nd order accuracy, but it can reach its 4th order accuracy with only two prediction iterations. Similarly, due to the low-order extrapolation,  LEQGRK only has 3rd order accuracy, and it can easily reach its 4th order accuracy with one prediction iteration. From Figure \ref{fig:CH-error}, we also see that the LEQDIRK-PC scheme with only three prediction iterations can reach similar accuracy as IEQDIRK proposed in \cite{GongEnergy}, while the LEQGRK-PC scheme only requires two prediction iterations. 

\begin{figure}
\center
\subfigure[$L^2$ and $L^\infty$ errors using the DIRK4th scheme]{
\includegraphics[width=0.4\textwidth]{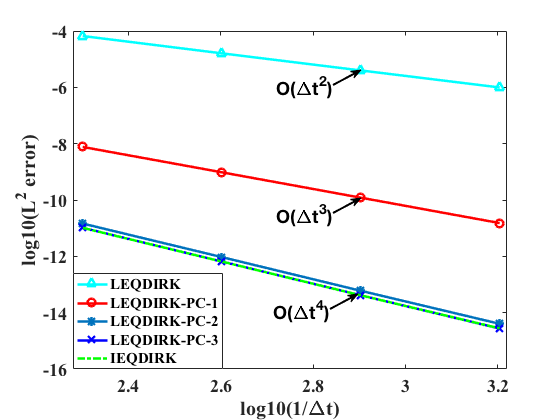}
\includegraphics[width=0.4\textwidth]{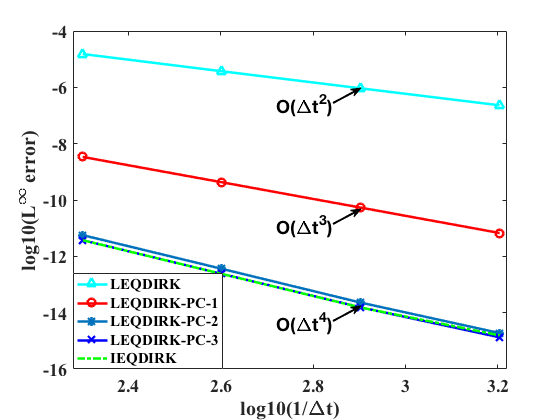}
}

\subfigure[$L^2$ and $L^\infty$ errors using the Gauss4th  scheme]{
\includegraphics[width=0.4\textwidth]{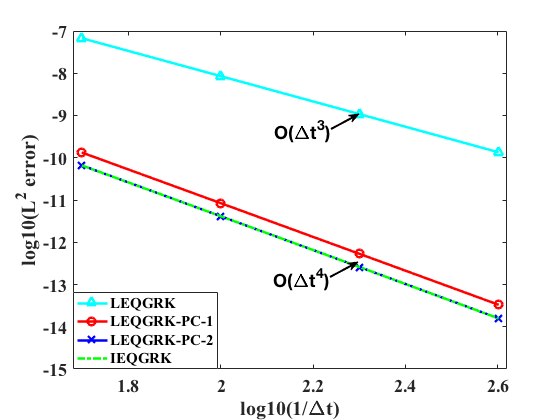}
\includegraphics[width=0.4\textwidth]{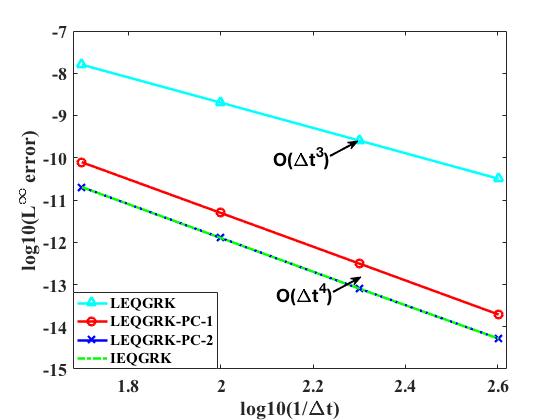}
}
\caption{Time step refinement tests with the proposed numerical schemes for the Cahn-Hilliard equation.}
\label{fig:CH-error}
\end{figure}

To further compare the DIRK4th and Gauss4th schemes, we summarize their $L^2$ and $L^\infty$ errors in the same plot, as shown in Figure \ref{fig:CH-error-2}(a)-(b), respectively. We observe that the Gauss4th scheme reaches its order of accuracy even with a larger time step size. After a few iterations, the DIRK4th scheme also reaches its order of accuracy quickly. Also, with the same time step size, the Gauss4th scheme is more accurate than the DIRK4th scheme.

\begin{figure}
\center

\subfigure[$L^2$ error]{\includegraphics[width=0.4\textwidth]{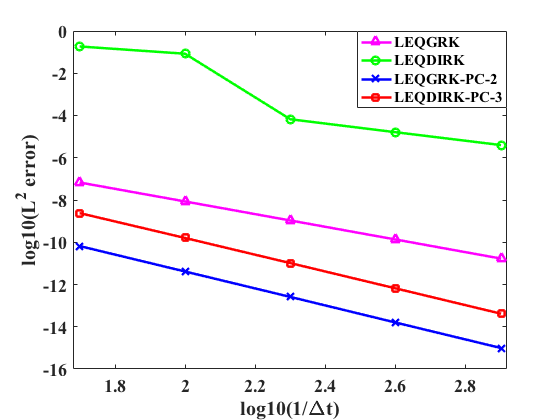}}
\subfigure[$L^\infty$ error]{\includegraphics[width=0.4\textwidth]{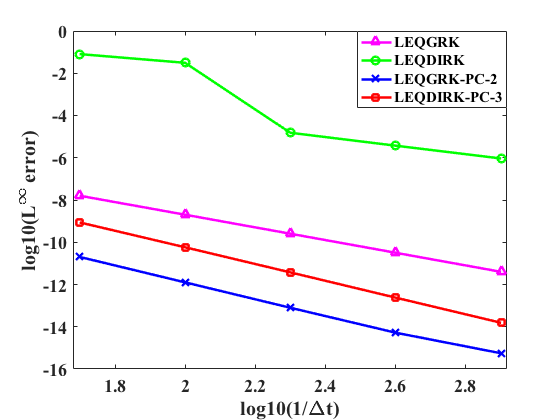}}
\caption{$L^2$ and $L^\infty$ errors using DIRK4th and Gauss4th schemes for the Cahn-Hilliard equation.}
\label{fig:CH-error-2}
\end{figure}

To further benchmark these two schemes,  we conduct several numerical tests.
For comparison, we also implement the widely used 2nd order convex splitting scheme (which we refer to as the 2nd-CS scheme in this paper),
\beq
\frac{\phi^{n+1}- \phi^n}{\Delta t} = \lambda \Delta \Big[-\varepsilon^2\Delta \phi^{n+\frac{1}{2}} +\frac{1}{2} ((\phi^n)^2+(\phi^{n+1})^2)\phi^{n+\frac{1}{2}}-(\frac{3}{2}\phi^{n}-\frac{1}{2}\phi^{n-1})\Big].
\eeq
We emphasis that there is no theoretical proofs for energy dissipation for the 2nd-CS scheme above, though it is more accurate than the first-order convex splitting scheme.

For the first example, we choose the domain as $[0 \,\,\, 1]^2$, and parameters $\lambda=1$, $\epsilon=0.01$, and $\gamma=1$. Then, we use the initial condition \cite{Wang&Wang&WiseDCDS2010}
\beq
\phi(x,y,t=0) = 0.05  \Big( \cos(3x)\cos(4y)+(\cos(4x)\cos(3y))^2 + \cos(x-5y)\cos(2x-y) \Big).
\eeq
This initial profile would drive a fast coarsening dynamics, such that the algorithm would predict 'wrong' dynamics if it is not accurate or robust enough. In this example, we intend to find the maximum possible time step that one can capture the correct dynamics numerically. Various numerical schemes with different time steps are implemented and compared. The numerical results are summarized in Figure \ref{fig:CH_stable}, where the predicted profile of $\phi(x,y)$ at $t=0.1$ are shown using different schemes and time-step sizes. We observe that the maximum possible time step size for the 2nd-order convex splitting scheme is approximately $\Delta t = 6.25 \times 10^{-5}$. For the DIRK4th scheme with 5 prediction iterations, the maximum time step size is approximately $\Delta t = 1.25 \times 10^{-4}$.  For the Gauss4th method with five prediction iterations, it is $2.5 \times 10^{-4}$.  Notice, the prediction steps could be easily solved with FFT, so the computational cost is negligible compared to the correction step.

These results indicate that the DIRK4th and Gauss4th schemes are superior over the 2nd order convex splitting scheme in this simulation. In addition, one should notice that there is no theoretical guarantee of monotonic energy decay with the 2nd order convex splitting scheme, and the implementation of the convex splitting scheme is relatively complicated, as nonlinear equations have to be solved at each time step. In contrast, the proposed high-order schemes here are linear and easy to implement. Also, they are rather general so that they can be applied to a broad class of gradient flow models.

\begin{figure}

\center
\subfigure[The profile of $\phi$ at $t=0.1$ using various time step sizes: $\Delta t= 2.5 \times 10^{-4}, 1.25 \times 10^{-4}, 3.125 \times 10^{-5}$ with the 2nd-order convex splitting scheme.]{
\includegraphics[width=0.22\textwidth]{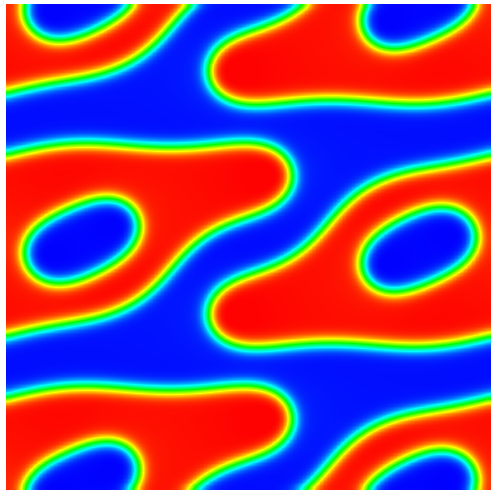}
\includegraphics[width=0.22\textwidth]{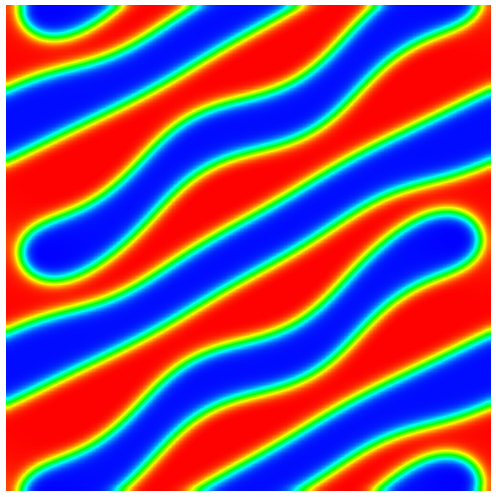}
\includegraphics[width=0.22\textwidth]{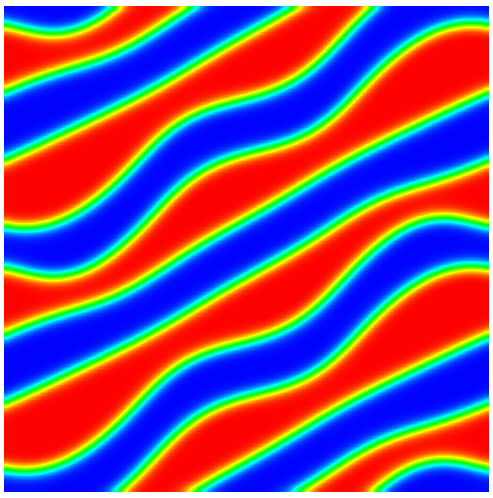}
}

\subfigure[The profile of $\phi$ at $t=0.1$ using various time step sizes: $\Delta t= 5 \times 10^{-4}, 2.5 \times 10^{-4}, 1.25 \times 10^{-4}$ with the  DIRK4th scheme.]{
\includegraphics[width=0.22\textwidth]{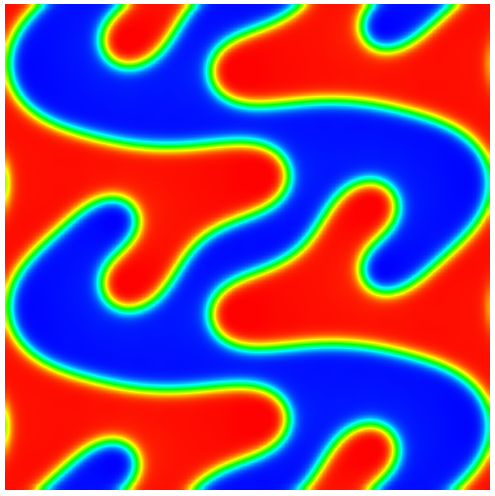}
\includegraphics[width=0.22\textwidth]{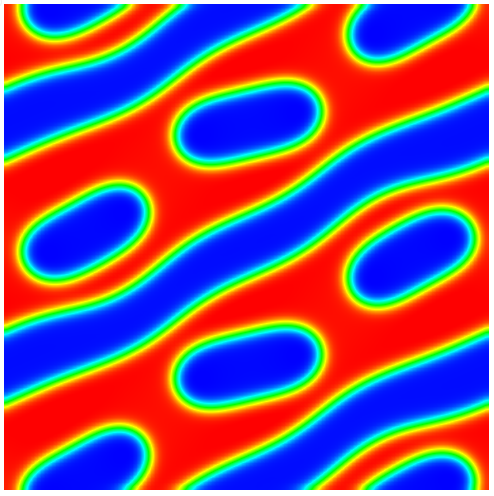}
\includegraphics[width=0.22\textwidth]{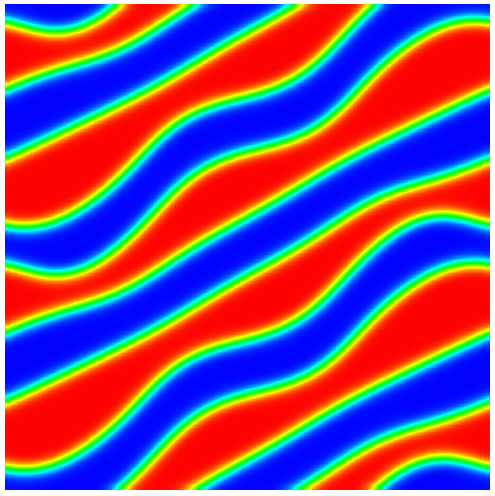}
}


\subfigure[The profile of $\phi$ at $t=0.1$ using various time step sizes: $\Delta t= 5 \times 10^{-4}, 2.5 \times 10^{-4}, 1.25 \times 10^{-4}$ with the Gauss4th scheme.]{
\includegraphics[width=0.22\textwidth]{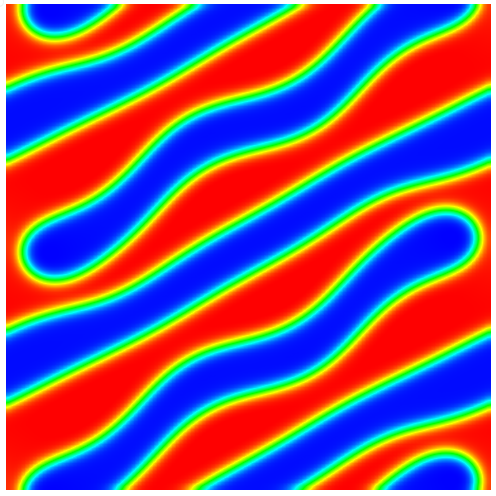}
\includegraphics[width=0.22\textwidth]{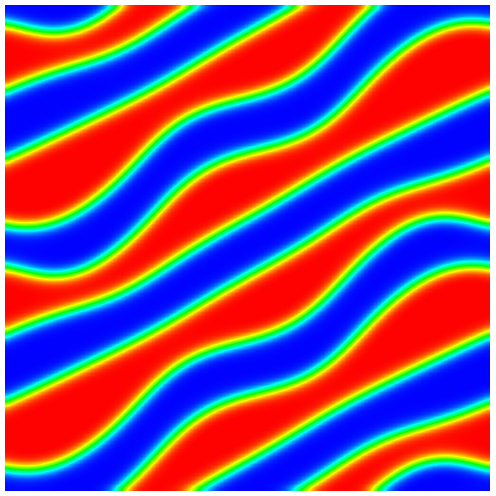}
\includegraphics[width=0.22\textwidth]{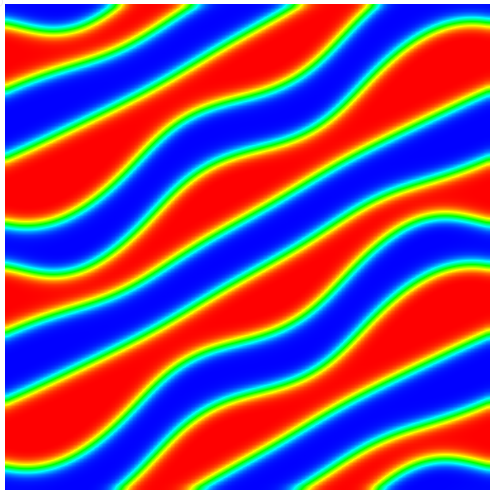}

}

\caption{A comparison of the three schemes on predicting accurate Cahn-Hilliard dynamics at various time step sizes. The figures show the numerical results of $\phi$ at time $t=0.1$ using different schemes with various time steps. The last sub-figure in each row indicates the maximum possible time step size to predict correct dynamics for the corresponding scheme. Both the DIRK4th scheme and the Gauss4th scheme perform better than the 2nd order CS scheme.}
\label{fig:CH_stable}
\end{figure}



To further confirm these findings,  we conduct an additional numerical experiment with random initial conditions. Specifically, we use
\beq
\phi(x,y,t=0) = 0.001 \textrm{rand}(x,y),
\eeq
where $\textrm{rand}(x,y) $ generates random numbers between $-1$ and $1$ uniformly. The rest settings are kept the same as in the previous example. The numerical results are summarized in Figure \ref{fig:CH_random}. This numerical example also indicates that the new schemes allow larger step sizes for accurately predicting the coarsening dynamics over the 2nd CS scheme.

\begin{figure}

\center
\subfigure[The profile of $\phi$ at $t=0.1$ using various time step sizes: $\Delta t= 2.5 \times 10^{-4}, 1.25 \times 10^{-4}, 6.25 \times 10^{-5}$ with the 2nd order convex splitting method]{
\includegraphics[width=0.22\textwidth]{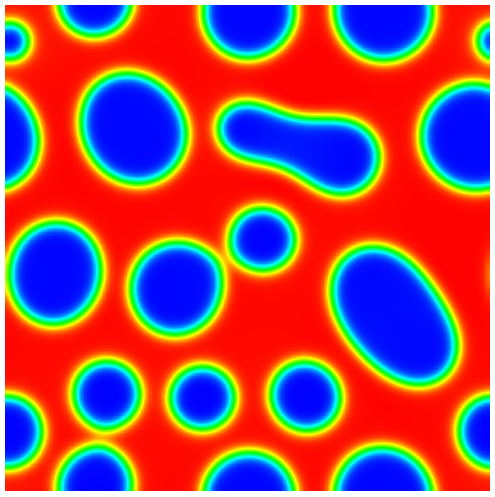}
\includegraphics[width=0.22\textwidth]{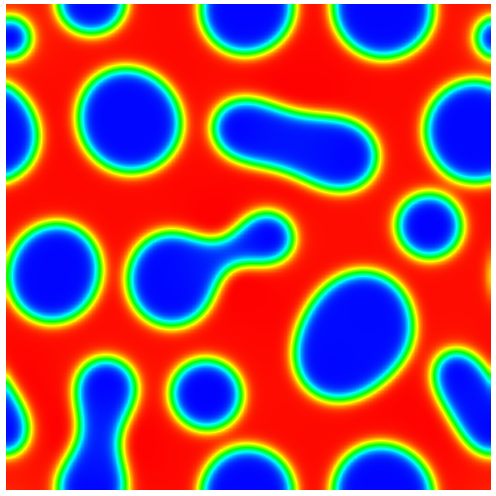}
\includegraphics[width=0.22\textwidth]{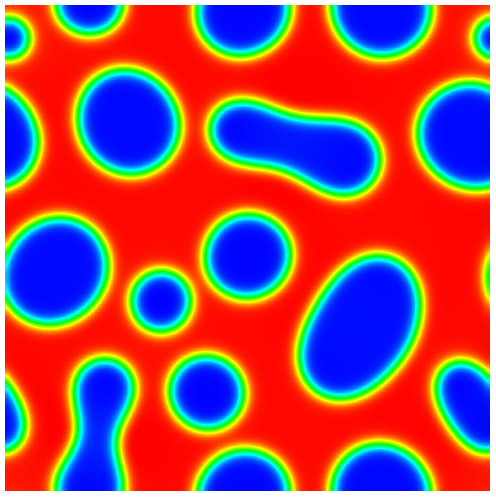}

}

\subfigure[The profile of $\phi$ at $t=0.1$ using various time step sizes: $\Delta t= 5 \times 10^{-4}, 2.5 \times 10^{-4}, 1.25 \times 10^{-5}$ with the DIRK4th method.]{
\includegraphics[width=0.22\textwidth]{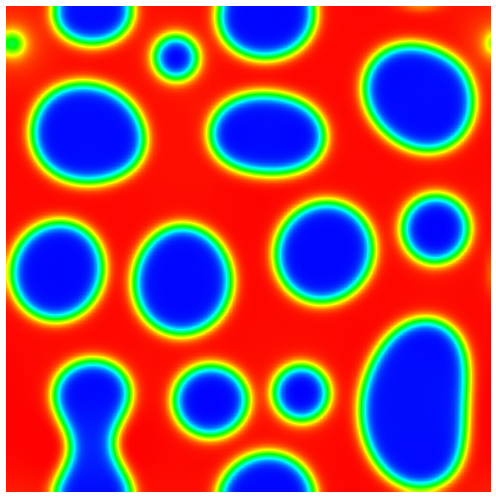}
\includegraphics[width=0.22\textwidth]{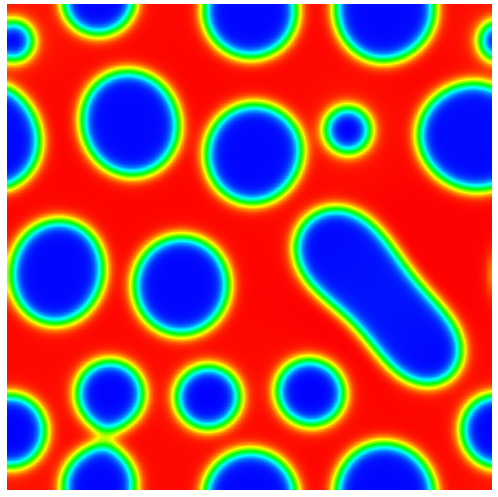}
\includegraphics[width=0.22\textwidth]{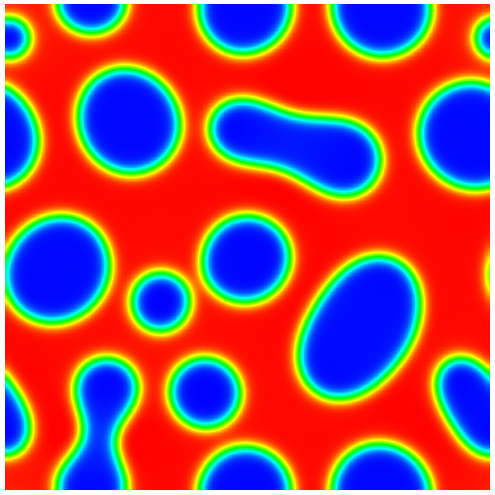}
}

\subfigure[The profile of $\phi$ at $t=0.1$ using various time step sizes: $\Delta t= 5 \times 10^{-4}, 2.5 \times 10^{-4}, 1.25 \times 10^{-5}$ with the Gauss4th method.]{
\includegraphics[width=0.22\textwidth]{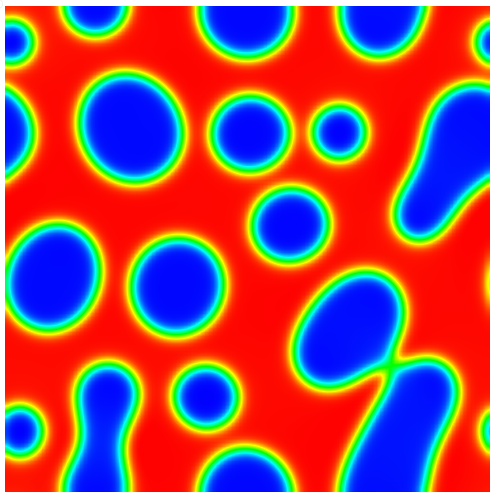}
\includegraphics[width=0.22\textwidth]{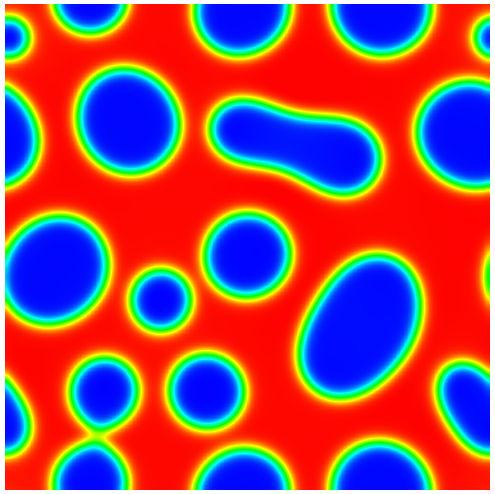}
\includegraphics[width=0.22\textwidth]{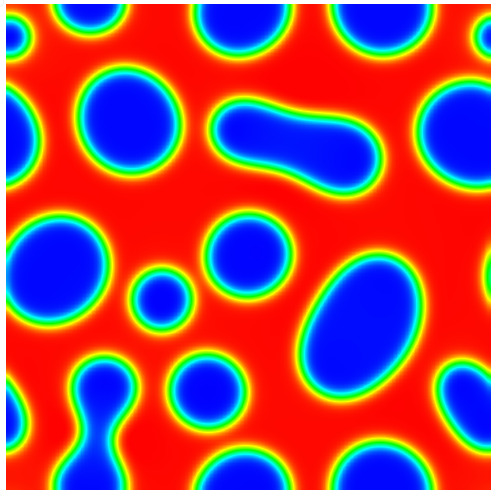}
}
\caption{A comparison of the three schemes on predicting accurate Cahn-Hilliard dynamics with random initial conditions. These figures show the numerical results of $\phi$ at time $t=0.1$ using different schemes with various time steps. The last sub-figure in each row indicates the maximum possible time step size to predict correct dynamics for the corresponding scheme. The Gauss4th scheme performs the best.}
\label{fig:CH_random}
\end{figure}

\subsection{Molecular beam epitaxial growth model}
In this subsection, we focus on the molecular beam epitaxial growth model with slope selection given as follows
\beq\label{MBE}
\phi_t  =-\lambda\Big( \varepsilon^2\Delta^2 \phi - \nabla \cdot \big( (|\nabla \phi|^2-1)\nabla \phi \big) \Big),
\eeq
where the free energy functional is given by 
\beq
F = \dfrac{\varepsilon^2}{2}\|\Delta \phi\|^2+\frac{1}{4} \big\||\nabla\phi|^2-1\big\|^2.
\eeq
We let $q = \frac{1}{2}\big(|\nabla\phi|^2-1-\gamma\big)$ and rewrite the energy functional as
\beq
\cF = \frac{1}{2} \Big(\phi, \varepsilon^2\Delta^2 \phi - \gamma\Delta\phi \Big) + \|q\|^2 - \frac{\gamma^2+2\gamma}{4}|\Omega|.
\eeq
Using the EQ reformulation, we have the following equivalent system
\beq\label{MBE-EQ}
\begin{cases}
\phi_t = -\lambda\Big( \varepsilon^2\Delta^2 \phi - \gamma \Delta \phi - \nabla \cdot \big( 2q\nabla \phi \big) \Big),\\
q_t = \nabla\phi\cdot \nabla\phi_t,
\end{cases}
\eeq
with the consistent initial condition
\beq\label{MBE-IC}
\begin{cases}
\phi(t=0) = \phi_0,\\
q(t=0) = \frac{1}{2}\big(|\nabla\phi_0|^2-1-\gamma\big).
\end{cases}
\eeq
It is readily to show that  new system \eqref{MBE-EQ} obeys the following energy dissipation law
\beq
\frac{d \cF}{d t} = -\lambda\Big\| \varepsilon^2\Delta^2 \phi - \gamma \Delta \phi - \nabla \cdot \big( 2q\nabla \phi \big)\Big\|^2\leq 0.
\eeq

Applying the LEQRK-PC scheme for system \eqref{MBE-EQ}, we have the following scheme. 

\begin{scheme} \label{scheme:LEQRK-PC-MBE}
Let $b_i$, $a_{ij}$ ($i,j = 1,\cdots,s$) be real numbers and let $c_i = \sum\limits_{j=1}^s a_{ij}$. For given $(\phi^n, q^n)$ and $\Phi_N(t_n+c_i \Delta t), Q_N(t_n+c_i \Delta t),  \forall i$, the following intermediate values are first calculated by the following prediction-correction strategy
\begin{enumerate}
\item  Prediction: we set $\Phi_i^{n,0} = \Phi_N(t_n+c_i \Delta t),$ $Q_i^{n,0} = Q_N(t_n+c_i \Delta t)$ and $M>0$ as a given integer. For $m=0$ to $M-1$, we compute $\Phi_i^{n,m+1}$, $k_i^{n,m+1}$, $l_i^{n,m+1}$, $Q_i^{n,m+1}$ using
\beq \label{pre_RK_stage-MBE}
\left\{
\bea{l}
\Phi_i^{n,m+1} =  \phi^n +  \Delta t \sum\limits_{j=1}^s a_{ij} k_j^{n,m+1},\\
k_i^{n,m+1} =  -\lambda\Big( \varepsilon^2\Delta^2 \Phi_i^{n,m+1} - \gamma \Delta \Phi_i^{n,m+1} - \nabla \cdot \big( 2Q_i^{n,m} \nabla \Phi_i^{n,m} \big) \Big),\\
l_i^{n,m+1} = \nabla \Phi_i^{n,m+1}\cdot \nabla k_i^{n,m+1},\\
Q_i^{n,m+1} = q^n +  \Delta t \sum\limits_{j=1}^s a_{ij} l_j^{n,m+1},
\eea
\right.
\quad i = 1,\cdots,s.
\eeq
Given the error tolerance $TOL>0$, if  $\max\limits_{i}\|\Phi^{n,m+1}_{i}-\Phi^{n,m}_{i}\|_\infty < TOL$, we stop the iteration and  set $\Phi^{n,*}_{i} = \Phi_i^{n,m+1}$; otherwise, we set $\Phi^{n,*}_{i} = \Phi_i^{n,M}$.

\item  Correction: for the predicted $\Phi^{n,*}_{i}$, we compute the intermediate values $\Phi_i^n, Q_i^n, k_i^n, l_i^n$ via
\beq \label{Cor_RK_stage-MBE}
\left\{
\bea{l}
\Phi_i^n =  \phi^n +  \Delta t \sum\limits_{j=1}^s a_{ij} k_j^n,\\
Q_i^n = q^n +  \Delta t \sum\limits_{j=1}^s a_{ij} l_j^n,\\
k_i^n = -\lambda\Big( \varepsilon^2\Delta^2 \Phi_i^{n} - \gamma \Delta \Phi_i^{n} - \nabla \cdot \big( 2Q_i^{n} \nabla \Phi_i^{n,*} \big) \Big), \\
l_i^n = \nabla \Phi_i^{n,*}\cdot \nabla k_i^{n},
\eea
\right.
\quad i = 1,\cdots,s.
\eeq
\end{enumerate}
Then $(\phi^{n+1}, q^{n+1})$ is updated via
\ben
&& \phi^{n+1} = \phi^n +  \Delta t \sum\limits_{i=1}^s b_i k_i^n,\\
&& q^{n+1} = q^n +  \Delta t \sum\limits_{i=1}^s b_i l_i^n.
\een
\end{scheme}

We apply the proposed arbitrarily high order schemes to solve MBE model \eqref{MBE}. We repeat the time step refinement test first. Here we use domain $[0 \,\,\, 2\pi]^2$ and choose parameters $\lambda=0.01$, $\gamma =1$ and $\varepsilon=1$. By adding the proper force term on the right-hand side of the equation, we create the real solution
\beq
\phi(x,y,t) = \sin(x)\sin(y) \cos(t).
\eeq
for the MBE model. Then we solve the modified model in the domain with a periodic boundary using the pseudo-spectral method for spatial discretization on $128^2$ meshes. The $L^2$ errors and $L^\infty$ errors using different schemes and various time steps are summarized in Figure \ref{fig:MBE-error}. Here we observe similar results, i.e., the DIRK4th scheme reaches 2nd order accuracy without prediction, but obtain 4th order accuracy with only two iteration steps for both the $L^2$ and $L^\infty$ norms. This is due to the low-order approximation for extrapolating the explicit terms. Analogously, the Gauss4th scheme is 3rd order accurate without any prediction steps and reaches 4th order accuracy with one iteration step.

\begin{figure}
\center

\subfigure[$L^2$ and $L^\infty$ errors using the DIRK4th scheme]{
\includegraphics[width=0.45\textwidth]{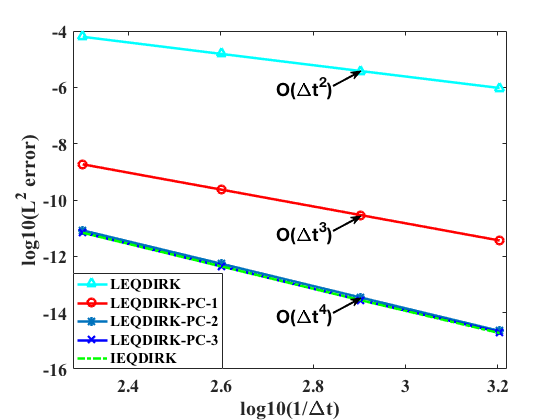}
\includegraphics[width=0.45\textwidth]{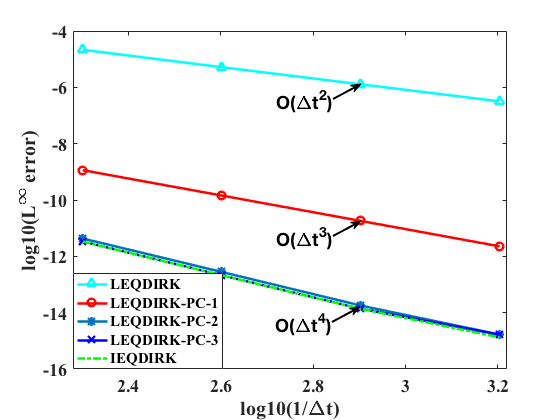}
}

\subfigure[$L^2$ and $L^\infty$ errors using the Gauss4th scheme]{
\includegraphics[width=0.45\textwidth]{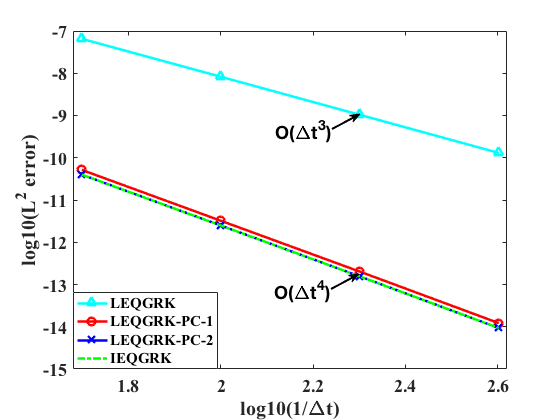}
\includegraphics[width=0.45\textwidth]{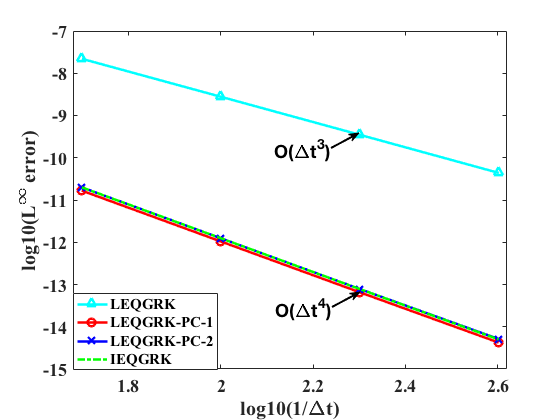}
}
\caption{Time step refinement tests with the proposed numerical schemes for the MBE model.}
\label{fig:MBE-error}
\end{figure}

To compare the accuracy of the DIRK4th scheme with that of the Gauss4th scheme in solving the molecular beam epitaxy (MBE) model, we summarize their $L^2$ and $L^\infty$ errors in the same figure as shown in \ref{fig:MBE-error-2}. We observe that the Gauss4th method has smaller errors than the DIRK4th scheme if using the same time step sizes.

\begin{figure}
\subfigure[$L^2$ errors]{
\includegraphics[width=0.45\textwidth]{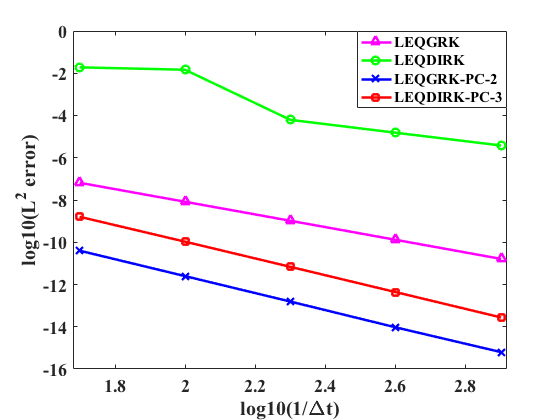}}
\subfigure[$L^\infty$ errors]{
\includegraphics[width=0.45\textwidth]{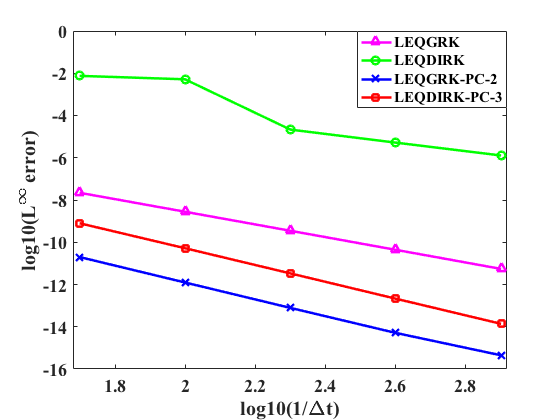}
}
\caption{$L^2$ and $L^\infty$ errors using the DIRK4th method and the Gauss4th method for the MBE model.}
\label{fig:MBE-error-2}
\end{figure}

Next, we use the proposed DIRK4th and Gauss4th schemes to solve two benchmark problems associated to the MBE model \eqref{MBE}. As before, we introduce the 2nd-order convex splitting scheme
\beq \label{eq:MBE-CS}
\frac{\phi^{n+1}-\phi^n}{\delta t} = -\lambda\Big( \varepsilon^2\Delta^2 \phi^{n+\frac{1}{2}} -  \frac{1}{2} \nabla \cdot \big( (|\nabla \phi^{n+1}|^2 +|\nabla \phi^n|^2) \nabla \phi^{n+\frac{1}{2}} \big) + \Delta (\frac{3}{2}\phi^n-\frac{1}{2}\phi^{n-1})  \Big),
\eeq
which will be used for comparison with the proposed linear high-order schemes.  

Following \cite{Wang&Wang&WiseDCDS2010}, we choose the domain as $[0 \,\,\, 2\pi]^2$, parameters $\lambda=1$, $\varepsilon^2=0.1$, and $\gamma = 1$.  We solve the MBE model in a periodic domain using the pseudo-spectral method with $128^2$ meshes. All the numerical schemes (i.e., the DIRK4th, Gauss4th, and 2nd CS scheme) are implemented. Five prediction iterations are used for both the DIRK4th and Gauss4th scheme.  The energy from $t=0$ to $t=15$ are calculated with different time steps and  the results are summarized in Figure \ref{fig:MBE-benchmark}.
We observe the maximum time steps to obtain accurate solutions are $\Delta t = 0.015625$ for 2nd order convex splitting scheme, $\Delta t = 0.0025$ for LEQDIRK-PC-5, and $\Delta t = 0.015625$ for LEQGRK-PC-5.

We emphasis that, using the 2nd-order convex splitting scheme  \eqref{eq:MBE-CS} for the MBE model, nonlinear equations have to be solved at each time step, but DIRK and Gauss scheme are all linear and easy to implement. In addition, for the 2nd order convex splitting scheme, there are no theoretical proofs for energy dissipation, but the high-order linear schemes introduced in this paper all guarantee energy dissipation laws.
\begin{figure}

\center
\subfigure[Convex splitting scheme]{
\includegraphics[width=0.31\textwidth]{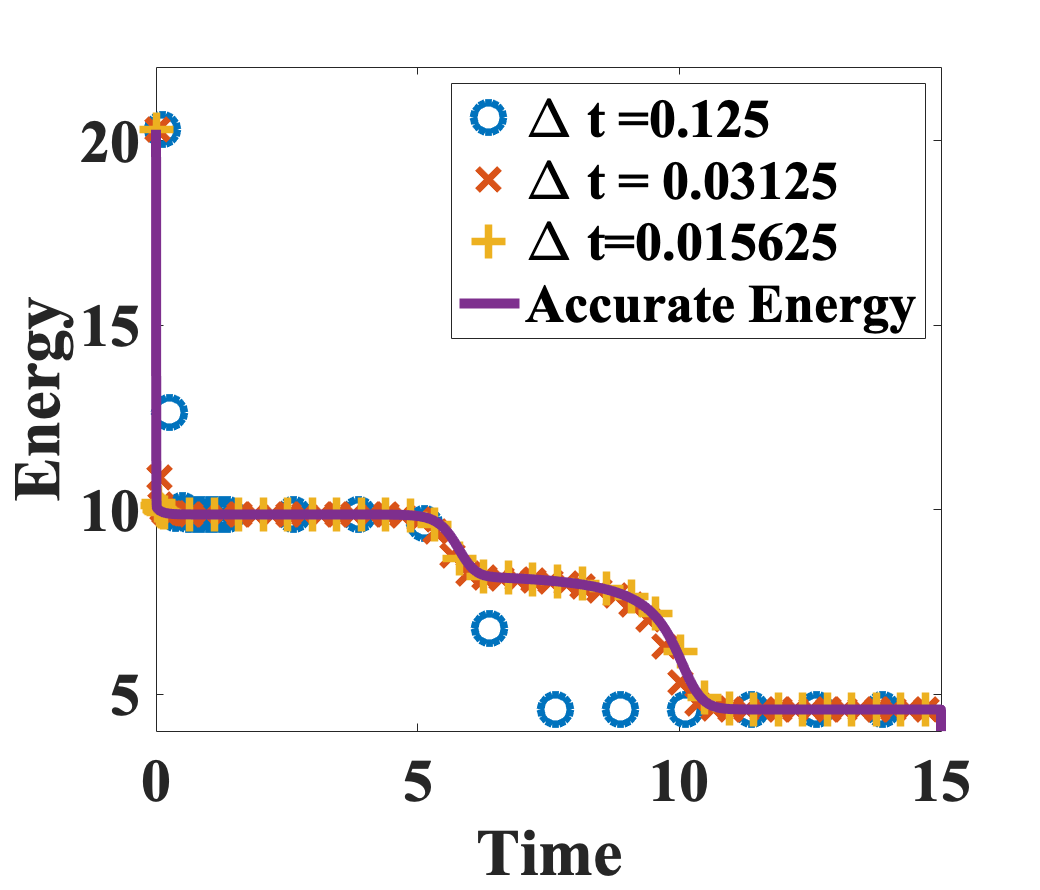}
}
\subfigure[LEQDIRK-PC-5]{
\includegraphics[width=0.31\textwidth]{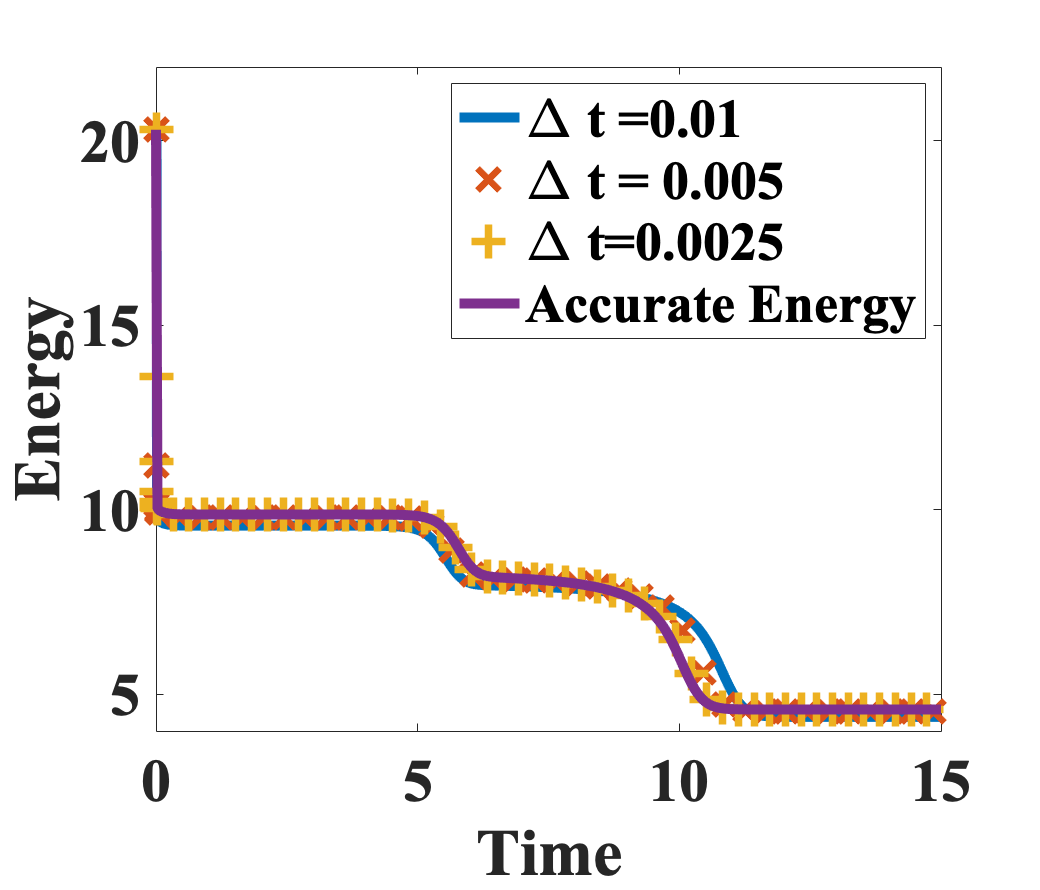}
}
\subfigure[LEQGRK-PC-5]{
\includegraphics[width=0.31\textwidth]{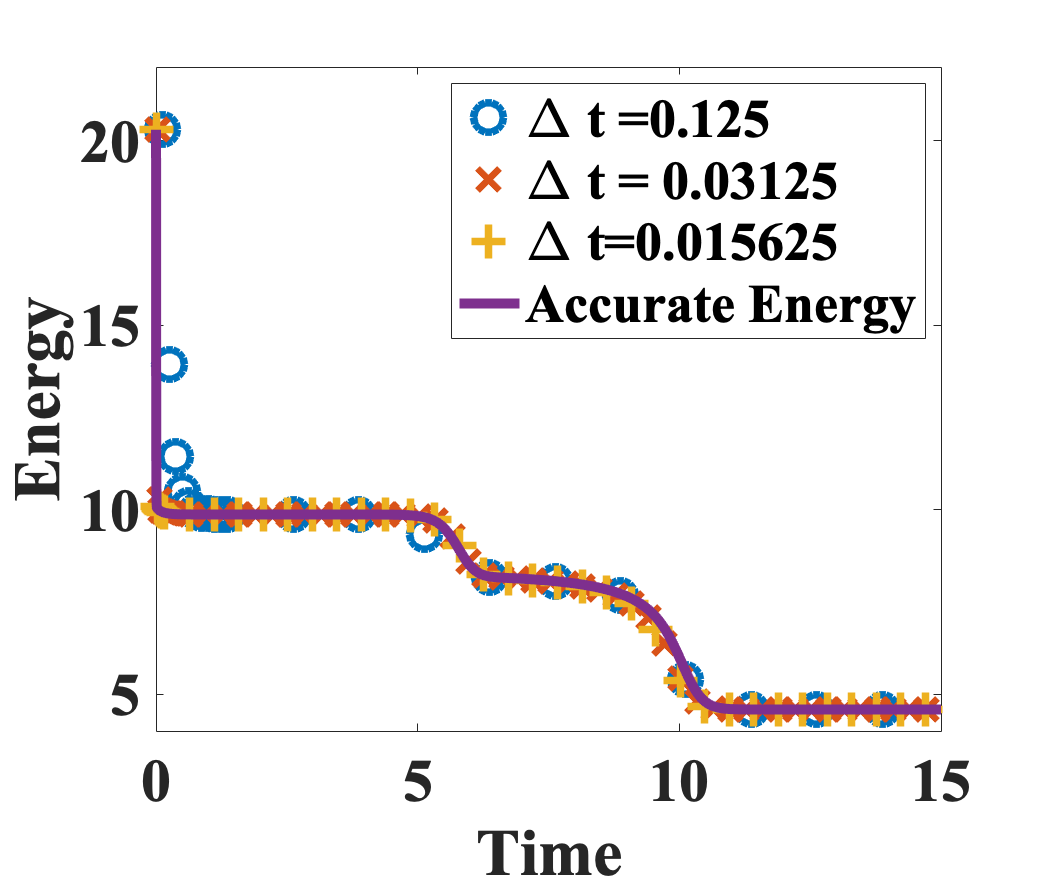}
}

\caption{Numerical results of energy evolution for the MBE model using different schemes with various time steps.}
\label{fig:MBE-benchmark}
\end{figure}

With the proposed high-order schemes, we can easily solve the MBE model with relatively larger time steps in most cases, making simulating long-time dynamics practical. Here we give an additional example as an illustration. We choose the domain as $[0 \,\,\, 12.8]^2$ and $\varepsilon=0.03$. The rest parameters are the same as in previous examples. We use $256^2$ meshes. It is known that the MBE coarsening dynamics follows a power law, where the energy decreases as $O(t^{-\frac{1}{3}})$, and the roughness increases as $O(t^{\frac{1}{3}})$ \cite{Wang&Wang&WiseDCDS2010}. The numerical results are summarized in Figure \ref{fig:MBE-long_time_e_r}, showing a strong agreement with the expected power law.
\begin{figure}

\center

\subfigure[Energy]{
\includegraphics[width=0.45\textwidth]{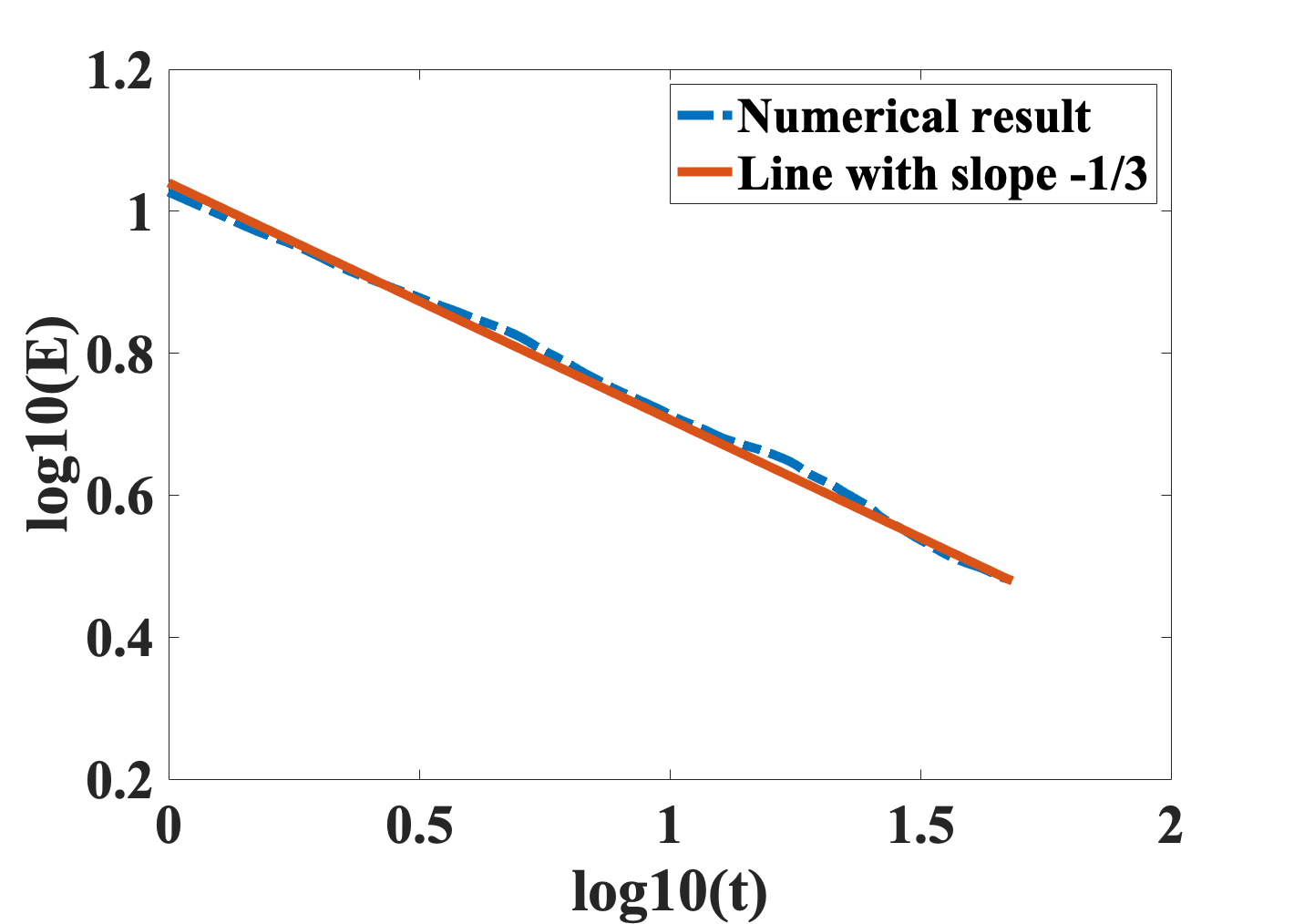}
}
\subfigure[Rougness]{
\includegraphics[width=0.45\textwidth]{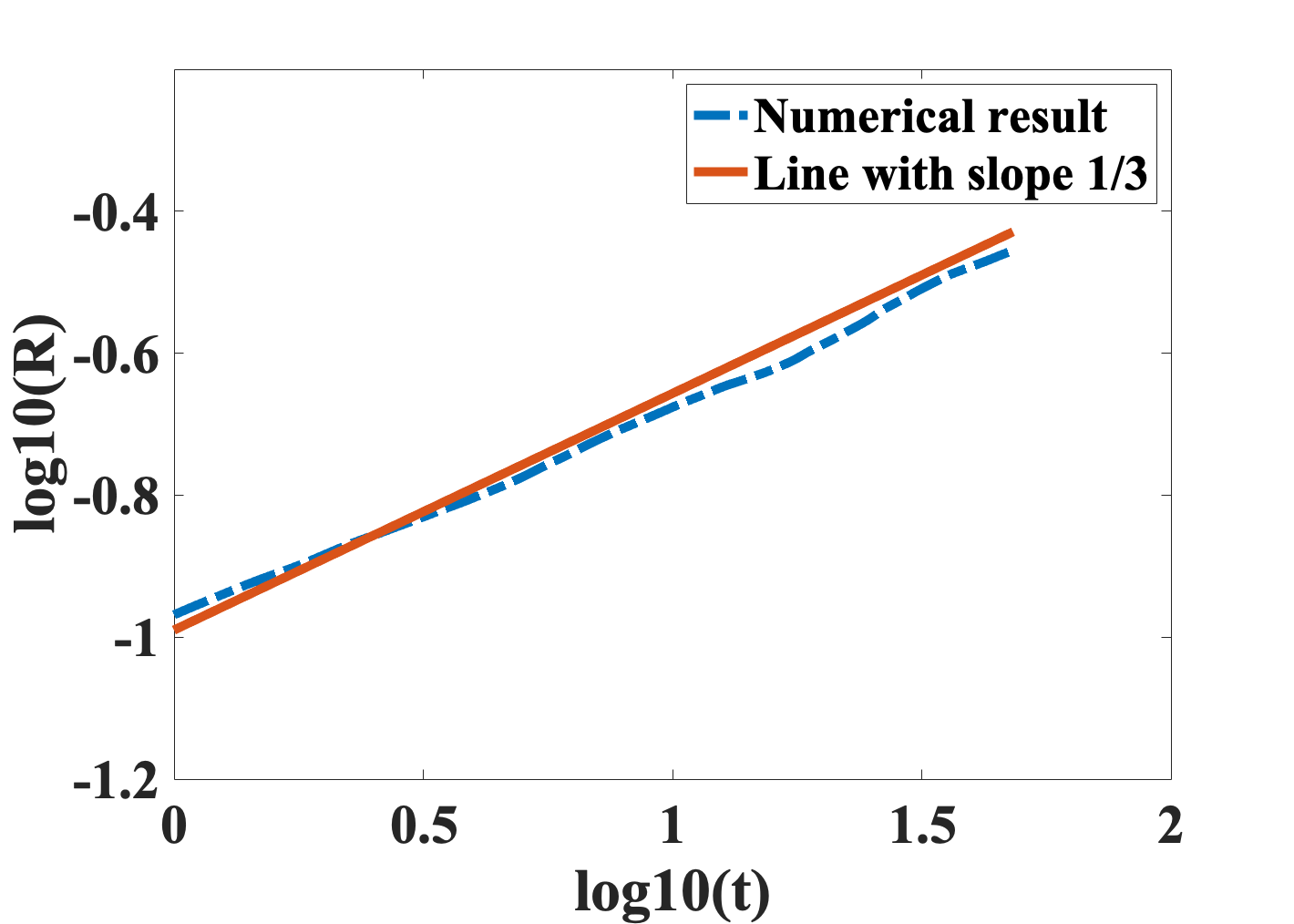}
}

\caption{The numerical results show proper power law dynamics for  the decreasing energy as $O(t^{-1/3})$ and increasing roughness as $O(t^{1/3})$.}
\label{fig:MBE-long_time_e_r}
\end{figure}

The profile of $\phi$ and $\Delta \phi$ at different times are summarized in  Figure \ref{fig:MBE-long_time_phi} and \ref{fig:MBE-long_time_lapphi}, respectively. These profiles look qualitatively similar to the reported results. These results strongly support our claim that the general arbitrarily high order linear schemes can be applied to predict accurate dynamics for the MBE model.

\begin{figure}
\center
\subfigure[$\phi$ at $t=0$]{\includegraphics[width=0.24\textwidth]{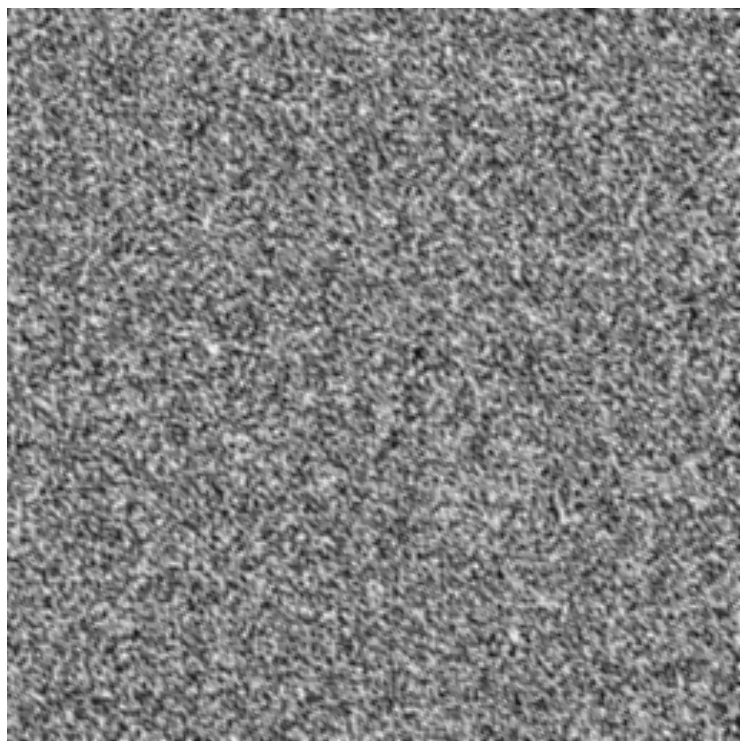}}
\subfigure[$\phi$ at $t=5$]{\includegraphics[width=0.24\textwidth]{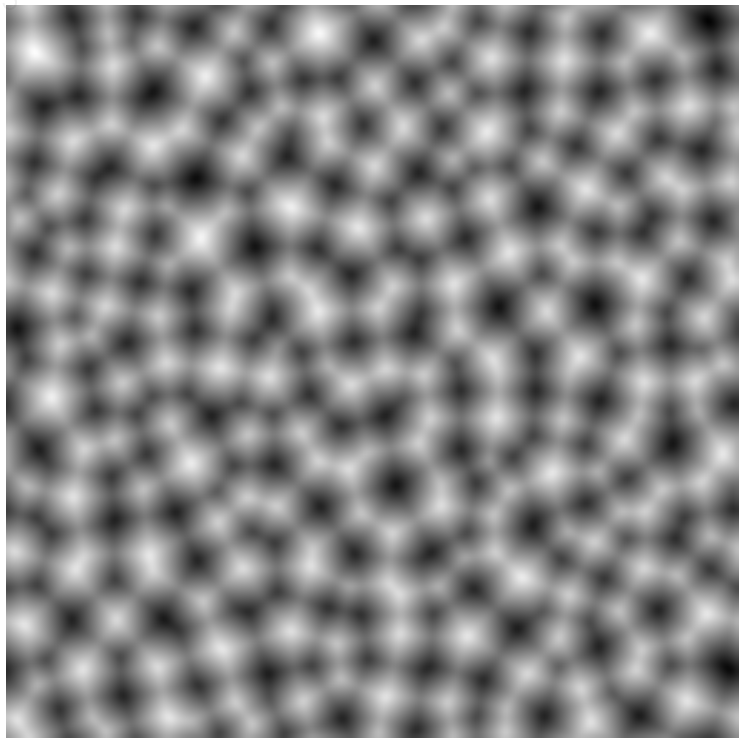}}
\subfigure[$\phi$ at $t=10$]{\includegraphics[width=0.24\textwidth]{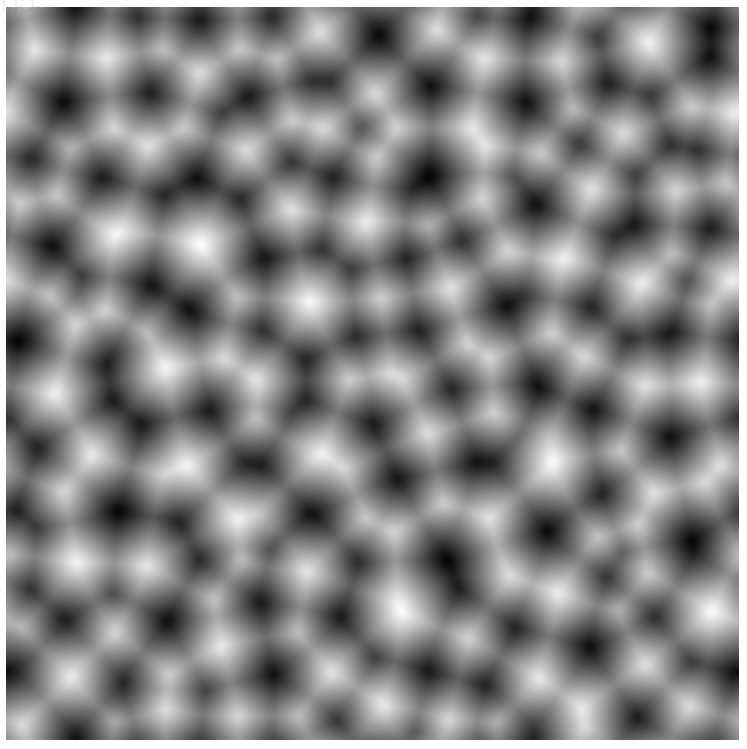}}

\subfigure[$\phi$ at $t=20$]{\includegraphics[width=0.24\textwidth]{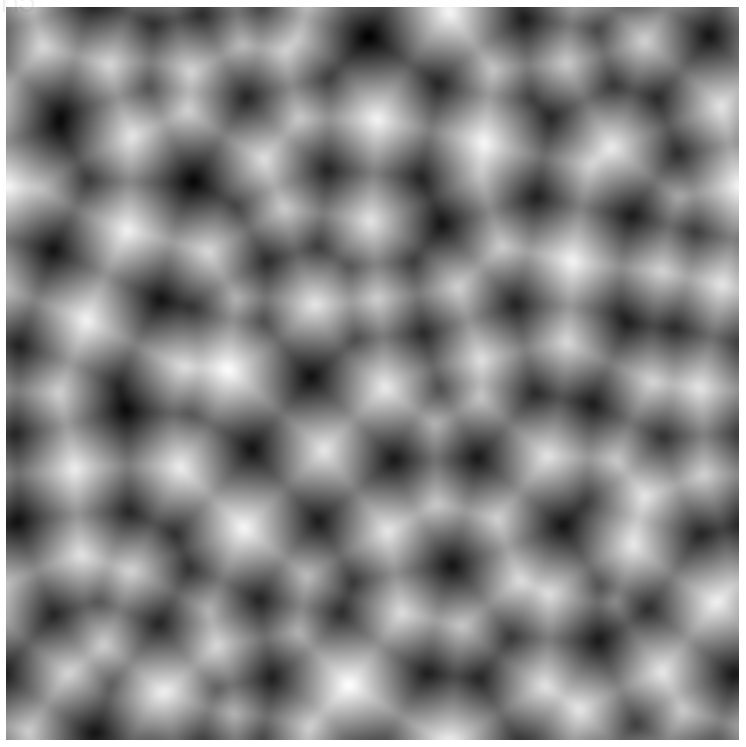}}
\subfigure[$\phi$ at $t=30$]{\includegraphics[width=0.24\textwidth]{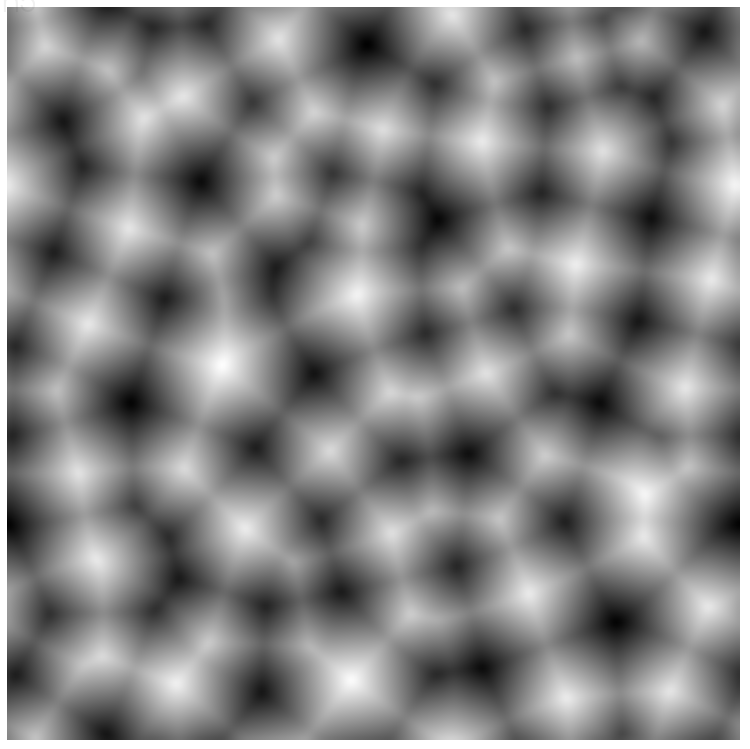}}
\subfigure[$ \phi$ at $t=45$]{\includegraphics[width=0.24\textwidth]{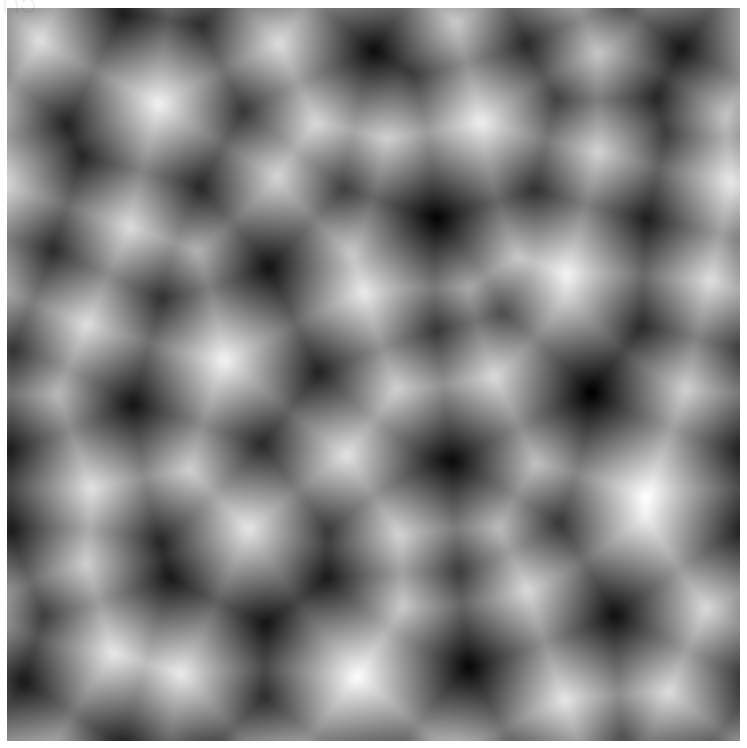}}
\caption{The profile of $\phi$ at different times.}
\label{fig:MBE-long_time_phi}
\end{figure}

\begin{figure}
\center
\subfigure[$\Delta \phi$ at $t=0$]{\includegraphics[width=0.24\textwidth]{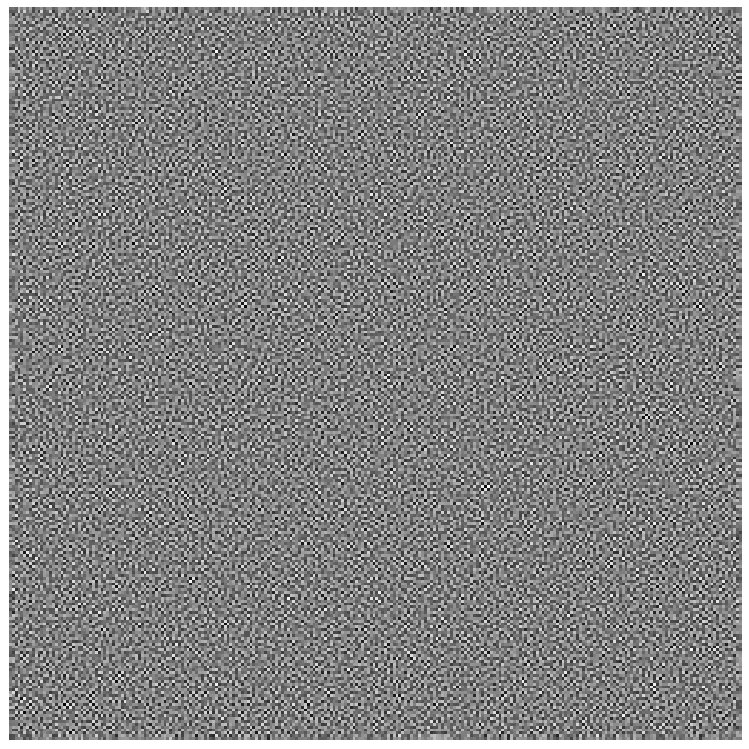}}
\subfigure[$\Delta \phi$ at $t=5$]{\includegraphics[width=0.24\textwidth]{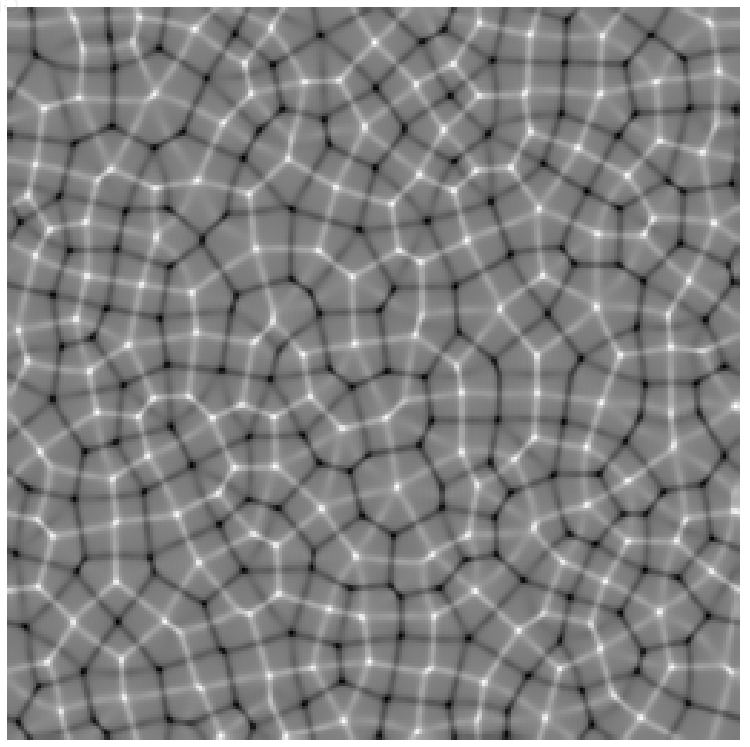}}
\subfigure[$\Delta \phi$ at $t=10$]{\includegraphics[width=0.24\textwidth]{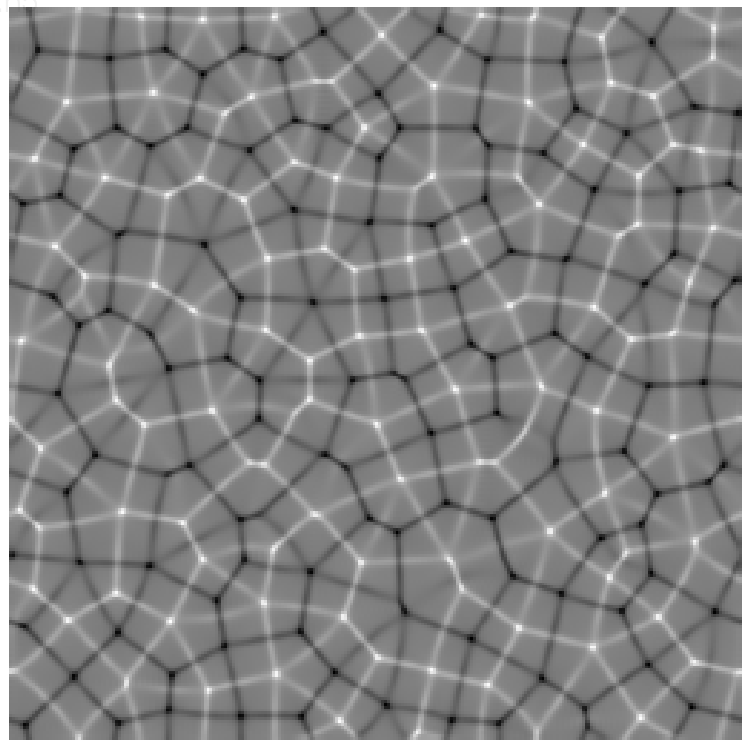}}

\subfigure[$\Delta \phi$ at $t=20$]{\includegraphics[width=0.24\textwidth]{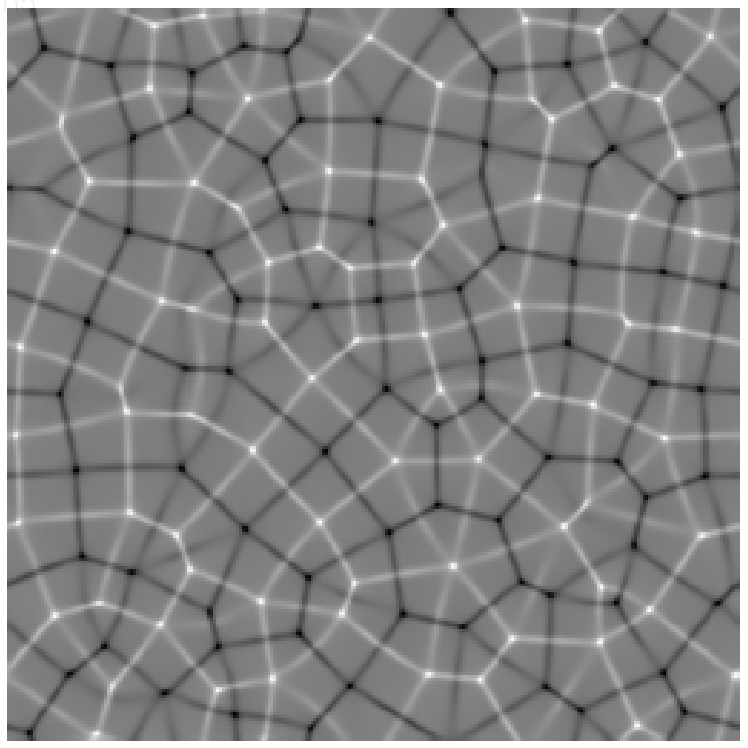}}
\subfigure[$\Delta \phi$ at $t=30$]{\includegraphics[width=0.24\textwidth]{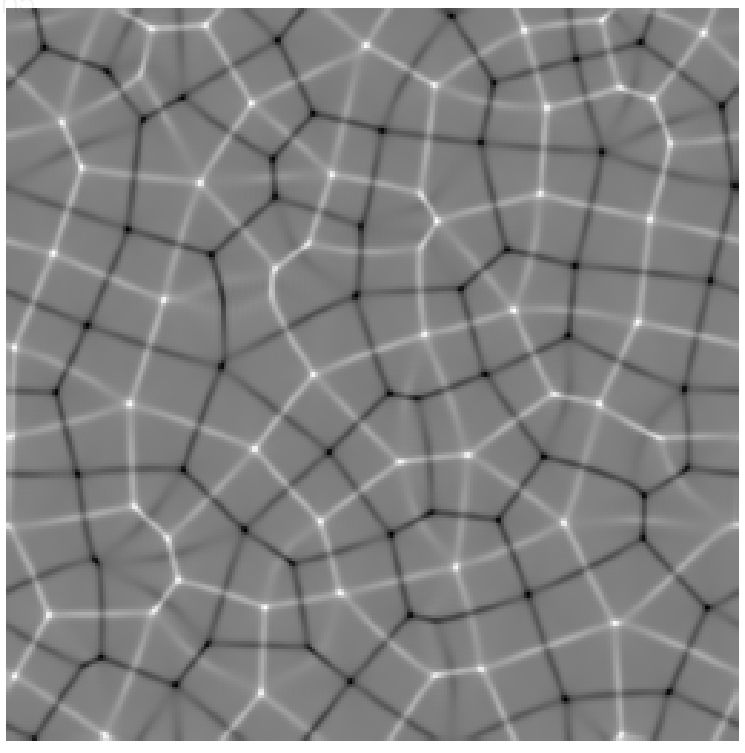}}
\subfigure[$\Delta \phi$ at $t=45$]{\includegraphics[width=0.24\textwidth]{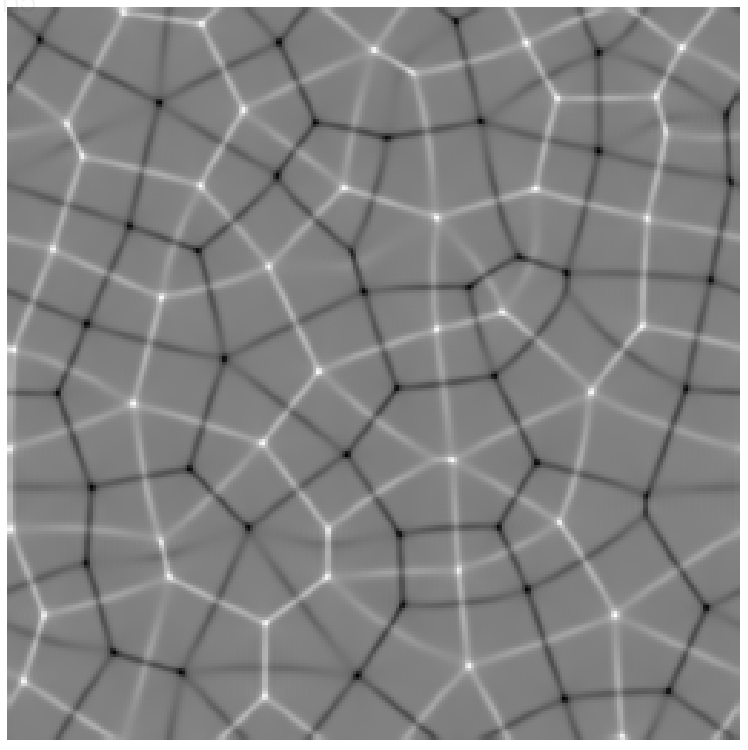}}

\caption{The profile of $\Delta \phi$ at different times.}
\label{fig:MBE-long_time_lapphi}
\end{figure}

\section{Conclusions}\label{sect:con}
In this paper, we present a new paradigm for developing arbitrarily high order, fully discrete numerical algorithms. These newly proposed algorithms have several advantageous properties: (1) the schemes are all linear such that they are easy to implement and computationally efficient; (2) the schemes are unconditionally energy stable and uniquely solvable such that large time steps can be used in some  long-time dynamical simulations; (3) the schemes can reach arbitrarily high-order of accuracy spatial-temporally such that relatively large meshes can  guarantee the desired accuracy of numerical solutions; (4) the schemes do not depend on the specific expression of the free energy explicitly such that it can be readily applied to a large class of general gradient flow models. The proofs for energy stability and uniquely solvability are given, and numerical tests with benchmark problems are shown to illustrate the effectiveness of the proposed schemes.

\section*{Acknowledgments}
Yuezheng Gong's work is partially supported by the Natural Science Foundation of
Jiangsu Province (Grant No. BK20180413) and the National Natural Science Foundation of China (Grant No. 11801269).
Jia Zhao's work is partially supported by National Science Foundation under grant number NSF DMS-1816783; and National Institutes of Health under grant number R15-GM132877.  Jia Zhao would also like to acknowledge NVIDIA Corporation for their donation of a  Quadro P6000 GPU for conducting some of the numerical simulations in this paper.
Qi Wang's work is partially supported by  DMS-1815921,  OIA-1655740 and a GEAR award from SC EPSCoR/IDeA Program.  Supports by
NSFC awards \#11571032, \#91630207, and NSAF-U1930402 are also acknowledged.


\end{document}